\documentclass[12pt]{amsart}

\usepackage{comment}
\usepackage{amsmath}
\usepackage{amsfonts,mathtools,upgreek}
\usepackage{amssymb,color}
\usepackage[hidelinks]{hyperref}
\usepackage{amsthm,doi}
\usepackage{thmtools,esint}
\usepackage{physics}
\usepackage{tikz}
\usepackage{enumitem}
\allowdisplaybreaks

\usepackage[a4paper, hmargin={2cm,2cm},vmargin={3cm,2.5cm}]{geometry}

\newcommand{\N}{\mathbb{N}}
\newcommand{\Z}{\mathbb{Z}}
\newcommand{\C}{\mathbb{C}}
\newcommand{\R}{\mathbb{R}}
\renewcommand{\S}{\mathcal{S}}
\newcommand{\V}{\mathfrak{V}}
\newcommand{\hh}{\mathbb{H}}
\newcommand{\HH}{\mathcal{H}}
\newcommand{\F}{\mathcal{F}}

\newcommand{\A}{A_\Omega}
\newcommand{\dual}{{g^\mathrm{d}_\Lambda}}
\newcommand{\ddual}{\varphi^\mathrm{d}}
\newcommand{\vr}{\varkappa(A)}

\DeclareMathOperator{\tra}{trace}

\DeclareMathOperator{\pp}{Per}
\DeclareMathOperator{\ii}{\nabla}
\DeclareMathOperator{\ee}{E}
\DeclareMathOperator{\vv}{Var}

\DeclareMathOperator{\gl}{Gl}
\DeclareMathOperator{\id}{Id}

\newtheorem{lemma}{Lemma}[section]
\newtheorem{theorem}[lemma]{Theorem}
\newtheorem{corollary}[lemma]{Corollary}
\newtheorem{proposition}[lemma]{Proposition}

\theoremstyle{definition}
\newtheorem{definition}[lemma]{Definition}

\newtheorem{remark}[lemma]{Remark}

\numberwithin{equation}{section}

\newcommand{\isof}{\iota}
\newcommand{\mf}{\vartheta_\varphi}
\newcommand{\Mf}{\Theta_\varphi}
\newcommand{\mg}{\vartheta_g}
\newcommand{\Mg}{\Theta_g}

\title[Spectral deviation of concentration operators on RKHS]{Spectral deviation of concentration operators on reproducing kernel Hilbert spaces}

\author[F. Marceca]{Felipe Marceca}
\address[F. M.]{Department of Mathematics, University College London \\
	 Gower Street \\
	London WC1E 6BT, UK}
\email{f.marceca@ucl.ac.uk}

\author[J. L. Romero]{Jos\'e Luis Romero}
\address[J. L. R.]{Faculty of Mathematics \\
	University of Vienna \\
	Oskar-Morgenstern-Platz 1 \\
	1090 Vienna, Austria \\and
	Acoustics Research Institute\\ Austrian Academy of Sciences\\Dominika- nerbastei 16, 1010 Vienna,  Austria}
\email{jose.luis.romero@univie.ac.at}

\author[M. Speckbacher]{Michael Speckbacher}

\address[M. S.]{ 
	Acoustics Research Institute\\ Austrian Academy of Sciences\\ Dominikanerbastei 16, 1010 Vienna,  Austria}
\email{michael.speckbacher@oeaw.ac.at}

\author[L. Valentini]{Lisa Valentini}

\address[L. V.]{ 
Department of Pure Mathematics and Mathematical Statistics
\\ University of Cambridge \\
Wilberforce Road, Cambridge CB3 0WA, United Kingdom}
\email{lv390@cam.ac.uk}

\begin{document}
\begin{abstract}
We study the eigenvalue profile of concentration operators (multiplication by an indicator function followed by projection) acting on reproducing kernel Hilbert spaces. The spectral profile of such operators provides a useful notion of local degrees of freedom. We formalize this idea by estimating the number of eigenvalues that lie away from 0 and 1, commonly referred to as the plunge region.

Our main motivation is to treat discrete and continuous settings simultaneously and uniformly, and to be able to argue that approximations arising from discretization schemes reflect, in a non-asymptotic sense, the spectral profile of their continuous counterparts. As a case in point, we show that Gabor multipliers computed on sufficiently fine grids obey spectral deviation estimates similar to those available for the short-time Fourier transform (STFT) with bounds that are  uniform in the discretization step. Concretely, this means that the theoretical localization properties of the STFT are observable in practice.
\end{abstract}
\maketitle

\section{Introduction}

\subsection{Concentration operators}
Let $(X,d,\mu)$ be a locally compact, $\sigma$-compact metric measure space $X$ and $\hh \subset L^2(X)$ a distinguished linear space of functions. We consider the question of quantifying the number of degrees of freedom that $\hh$ has per unit space. A naive answer can be given in terms of the dimensions of the restricted spaces $\{f \cdot 1_\Omega\}$, where
$\Omega \subset X$ is a domain and $1_\Omega$ is its indicator function, but this answer is too simplistic as these dimensions are infinite in many interesting cases. In this article, we look into a more refined form of quantifying local degrees of freedom, formulated in terms of the \emph{concentration operator} 
\begin{align}\label{eq_intro_co}
T_\Omega f := P (f \cdot 1_\Omega)\,, \qquad f \in \hh\,,
\end{align}
which is defined by multiplication by an indicator function followed by orthogonal projection onto $\hh$. It is easy to see that $T_\Omega$ is positive semi-definite and contractive. While $T_\Omega$ may fail to have finite rank, for many functional spaces $\hh$,
$T_\Omega$ is known to be close to an orthogonal projection, in the sense that its singular
values transition rapidly from almost 1 to almost 0. We shall investigate this phenomenon in general, and start by discussing some important examples.

\subsection{Time-frequency localization}\label{sec_intro_tf}
Our main motivation comes from \emph{time-frequency analysis}, where a function $f \in L^2(\mathbb{R}^d)$ is described by means of its time and frequency correlations with a smooth and fast-decaying \emph{window function} $g:\mathbb{R}^d \to \mathbb{C}$:
\begin{align}\label{eq_intro_invf}
V_g f(x,\xi) = \int_{\mathbb{R}^d} f(t)\overline{g(t-x)} e^{-2\pi i \xi t}\,dt\,, \qquad (x,\xi) \in \mathbb{R}^d\times\mathbb{R}^d\,.
\end{align}
The function $V_g f$ is called the \emph{short-time Fourier transform} of $f$ and, under the normalization $\|g\|_2=1$, provides the following representation:
\begin{align}\label{eq_tf_recon}
f (t) =
\int_{\mathbb{R}^d\times\mathbb{R}^d} V_g f(x,\xi)\, g(t-x)e^{2\pi i \xi t}\, dxd\xi\,, \qquad t \in \mathbb{R}^d\,,
\end{align}
where the integral converges in quadratic mean.
The range-space of the short-time Fourier transform $\hh=V_g L^2(\mathbb{R}^d)$ is a reproducing kernel subspace of $L^2(\mathbb{R}^{2d})$. For a domain $\Omega \subset \mathbb{R}^{2d}$, the concentration operator \eqref{eq_intro_co} implements the following \emph{time-frequency filter}: 
\begin{align}\label{eq_tf_dau}
\A^gf(t)=V_g^{-1} T_\Omega V_g f (t) =
\int_{\Omega} V_g f(x,\xi)\, g(t-x)e^{2\pi i \xi t}\, dxd\xi, \qquad t \in \mathbb{R}^d.
\end{align}
While perfect time-frequency restriction to $\Omega$ is impossible due to the uncertainty principle, the \emph{time-frequency localization operator} \eqref{eq_tf_dau}
was introduced in \cite{da88} as a suitable approximate restriction operator. 
Except for the multiplicity of the eigenvalue $0$, the spectra of $T_\Omega$ and $\A^g$ coincide. The deviation of the
\emph{spectral distribution function}
\begin{align}
\tra(1_{(\delta,1]}\A^g)=
\tra(1_{(\delta,1]}T_\Omega)=
\#\{\lambda \in \sigma(T_\Omega): \lambda > \delta\}
\end{align}
from $|\Omega|$, the Lebesgue measure of $\Omega$, quantifies how close one can get to the ideal time-frequency restriction.

The concentration operator $T_\Omega$ is a \emph{Toeplitz operator} with an indicator function as symbol, and its spectral asymptotics are classically studied under increasingly large isotropic dilations $\Omega \mapsto R\Omega$ valid asymptotically for a fixed spectral threshold $\delta$. Precise two-term {S}zeg\H{o} asymptotics for \eqref{eq_tf_dau} with Schwartz window $g$ and a $\mathcal C^2$ domain $\Omega$ were derived in \cite{MR3433287} and show that
\begin{align}\label{eq_intro_o}
\tra(1_{(\delta,1]}T_{R\Omega}) = 
\abs{ \Omega}\cdot R^{2d} + A(g,\partial\Omega,\delta) \cdot R^{2d-1} + o_{\delta,\Omega,g}(R^{2d-1}), \mbox{ as }R \longrightarrow \infty\,,
\end{align}
for a certain constant $A(g,\partial\Omega,\delta)$ defined explicitly in terms of the Wigner distribution of $g$ and the boundary of $\Omega$. Importantly, the little-$o$ term in \eqref{eq_intro_o} depends implicitly, and in an unspecified manner, on the spectral threshold $\delta$. 

However, some applications -- notably in mathematical physics \cite{AR23,AMR25,MR2} -- require \emph{two-parameter} estimates, that is, estimates with an explicit control on both $\delta$ and $\Omega$, that allows the two to vary together. For windows $g$ in the so-called \emph{Gelfand-Shilov class} $\mathcal{S}^{\beta,\beta}$ of functions with fast time and frequency decay \cite{MR1732755}, our recent work \cite{MR} provides the non-asymptotic estimate
\begin{align}\label{eq_intro_mr}
\big|\tra(1_{(\delta,1]}T_{\Omega}) - \abs{\Omega}\big| \leq C_{g,\partial \Omega} \cdot \HH(\partial \Omega) \cdot \log^*(\tau)^{2d\beta} \log^*\log^*(\tau), \qquad \tau=\max\{\tfrac{1}{\delta},\tfrac{1}{1-\delta}\},
\end{align}
where $\log^*(x)=\max\{1,\log(x)\}$,
$\HH(\partial \Omega)$ is the perimeter of $\Omega$
and $C_{g,\partial \Omega}$ is a constant that depends explicitly on the curvature of $\Omega$ (in a measure theoretic sense) and the time-frequency decay of $g$. Applying \eqref{eq_intro_mr} to the scaled domain $R\Omega$ one obtains estimates compatible with 
\eqref{eq_intro_o}, but valid in the two parameter regime --- see also \cite{MR} for a more technical version of \eqref{eq_intro_mr}, with an improved dependence on $\beta$ in the dilation regime.

\subsection{Spectral stability under discretization}
In practice, the time-frequency filter \eqref{eq_tf_dau} is approximately implemented by means of discrete analogs of \eqref{eq_intro_invf} known as \emph{Gabor frame expansions}:
\begin{align}\label{eq_intro_gab}
f(t) = \sum_{n,m \in \mathbb{Z}^d} V_g f\big(\tfrac{n}{N},\tfrac{m}{N}\big) \, g^{\mathrm{d}}_N\big(t- {\tfrac{n}{N}}\big) e^{2\pi i  {\tfrac{m  }{N}t}}, \qquad t \in \mathbb{R}^d.
\end{align}
Here, $N$ is a \emph{resolution parameter} and $g^{\mathrm{d}}_N \in L^2(\mathbb{R}^d)$ is the so-called \emph{canonical dual window} of $g$, obtained by solving a certain least-squares problem. The \emph{Gabor multiplier}
\begin{align}\label{eq_intro_gm}
M_{g,N,\Omega} f(t) = \sum_{n,m \in \mathbb{Z}^d} V_g f\big(\tfrac{n}{N},\tfrac{m}{N}\big) 1_\Omega\big(\tfrac{n}{N},\tfrac{m}{N}\big) \, g^{\mathrm{d}}_N\big(t- {\tfrac{n}{N}}\big) e^{2\pi i  {\tfrac{m  }{N}t}}\,, \qquad t \in \mathbb{R}^d\,,
\end{align}
defined by restricting the Gabor expansion \eqref{eq_intro_gab} to only those lattice points lying in the domain $\Omega$, is the preferred numerical implementation of \eqref{eq_tf_dau}.

The motivating question for our work is whether Gabor multipliers satisfy spectral deviation estimates analogous to \eqref{eq_intro_mr}, that are also stable under refinements of the sampling lattice $\Lambda$, and faithful to the continuous limit. In other words, \emph{can we see the spectral deviation estimate \eqref{eq_intro_mr} in practice?} This question is subtle and continues to receive attention from practitioners \cite{Ha}. While the Gabor multiplier \eqref{eq_intro_gm} converges in operator norm to the time-frequency localization operator \eqref{eq_tf_dau}, this fact is insufficient to imply (two-parameter) spectral deviation estimates for Gabor multipliers.

As an application of our main results, we shall see, for example, that for windows $g$ in the Gelfand-Shilov class $\mathcal{S}^{\beta,\beta}$ there exists $N_0$ such that for $N\ge N_0$:
\begin{align}\label{eq_intro_gm_spec}
\big|\tra(1_{(\delta,1]}M_{g,N,\Omega}) - \abs{\Omega}\big| \leq C_{g,\partial \Omega} \cdot \HH(\partial \Omega) \cdot \log^*(\tau)^{2d\beta} \log^*\log^*(\tau)\,, \qquad \tau=\max\{\tfrac{1}{\delta},\tfrac{1}{1-\delta}\}\,,
\end{align}
where, as before, $C_{g,\partial \Omega}$ is a constant that depends explicitly on the curvature of $\Omega$ and the time-frequency decay of $g$. Importantly, the right-hand side of \eqref{eq_intro_gm_spec} does not depend on the discretization parameter $N$. We also obtain similar estimates for windows $g$ in modulation spaces \cite{benyimodulation} that provide a discrete analogue to \cite[Theorem~1.4]{MR}, without discretization artifacts.

\subsection{Further applications}
The goal of treating discrete and continuous settings simultaneously and uniformly naturally leads us to work with \eqref{eq_intro_co} in the more general context of reproducing kernel Hilbert spaces. Beyond time-frequency analysis, our work implies the preservation of certain spectral properties under discretization of general frames (via frame multipliers) and concerns also vector-valued contexts (relevant in quantum harmonic analysis). While, at first glance, the quality of our estimates directly depends on the off-diagonal decay of the associated reproducing kernel, it is possible to combine our main result with certain decomposition methods and also treat slowly decaying kernels. As a proof of concept, we 
shall reinterpret the wave-packet expansion of \cite{israel15} as a means to
decompose the \emph{sinc kernel} into kernels that fall into the scope of this work; see \cite{MR4711850,kul2} for a related technique.

We now discuss our results in more detail.

\section{Results}
\subsection{Setup}
Let $X$ be a locally compact metric space, and let $\mu$ be a positive, $\sigma$-finite, Borel measure on $X$ that is finite on compact sets. 
To cover certain interesting examples, we shall work with vector-valued functions. Let $\HH$ be a separable Hilbert space and define $L^2(X,\HH)$ as the space of (weakly measurable) vector-valued functions for which 
$$
\|f\|^2_{L^2(X,\HH)}:=\int_X\|f(x)\|^2\,d\mu(x)
$$
is finite. Let $\hh \subseteq L^2(X, \HH)$ be a reproducing kernel Hilbert space (RKHS) of vector-valued functions with reproducing kernel $K \colon X\times X \to \S_2(\HH)$, where $\S_p=\S_p(\HH)$ are the Schatten $p$ operators on $\HH$, and write $K_y(x)=K(x,y)$ --- see Section~\ref{sec_vvrkhs} for details.

We shall study the spectrum of the concentration operator \eqref{eq_intro_co}. For technical reasons, it will be convenient to extend $T_\Omega$ by $0$ on $L^2(X,\HH) \ominus \hh$, that is,
\begin{equation}\label{deftoep}
    T_\Omega f = P(1_\Omega \cdot Pf) \,, \qquad f\in L^2(X,\HH)\,,
\end{equation}
where $\Omega\subseteq X$ is compact and $P \colon L^2(X,\HH) \to \hh$ is the orthogonal projection. 

\subsection{Assumptions}
We assume the following conditions:
\begin{enumerate}[label={\bf[C\arabic*]},ref={\bf[C\arabic*]}, leftmargin=4em]
    \item \label{norm2} $\lVert K_x \rVert_{L^2(X,\S_2)} = 1$ for every $x\in X$. 
    \item \label{stripe} There exists a constant $\gamma >0$ --- called \emph{inflation rate} --- such that
    \begin{align}\label{eq_ii}
    \ii(\Omega) := \sup_{n \in \mathbb{N}_0, E=\Omega,\Omega^c}
    2^{-\gamma (n-1)}
    \mu(\{x\in E^c : \ \mathrm{d}(x,E)\le 2^n\}) < \infty. 
    \end{align}
\end{enumerate}
In addition, sometimes we also assume: 
\begin{enumerate}[label={\bf[C\arabic*]},ref={\bf[C\arabic*]}, leftmargin=4em]
\setcounter{enumi}{2}
    \item \label{doubling} The measure $\mu$ is \emph{doubling}, that is, there exists a constant $C_{X}\ge 1$ such that for every $x\in X$ and $r>0$,
    \[0<\mu(B_{2r}(x))\le C_{X}\, \mu(B_{r}(x))\,.\]
\end{enumerate}

\noindent $\bullet$ Condition~\ref{norm2} is a normalization assumption. If not directly satisfied, one can renormalize the background measure as follows. Let $m(x):= \|K_x\|_{L^2(X,\S_2)}$,
$\tilde{X}=\{x \in X\,:\, m(x) \not=0\}$, and $d\widetilde{\mu}(x)=m(x)^2d\mu(x)$.
Then $\widetilde{\mathbb{H}}:=\big\{\tfrac{1}{m}\, F|_{\tilde{X}}\,:\, F\in\mathbb{H}\big\}$ is a RKHS as a subspace of $L^2(\tilde{X},\widetilde{\mu},\HH)$ with kernel $\widetilde{K}(x,y)=(m(x) m(y))^{-1}K(x,y)$, which satisfies \ref{norm2}, while $\widetilde{\mu}(X \smallsetminus \tilde{X})=0$ and $\mathbb{H} \ni F \mapsto \tfrac{1}{m} F|_{\tilde{X}} \in \widetilde{\mathbb{H}}$ is an isometric isomorphism.

\noindent $\bullet$ Condition~\ref{stripe} controls the inflation rate $\gamma$ of the set $\Omega$ and its complement; the 
constant $\ii(\Omega)$ defined in \eqref{eq_ii} is called the \emph{inflation constant} of $\Omega$. When $\Omega$ is a subset of $\mathbb{R}^d$ with smooth boundary, then $\gamma=d$ and $\ii(\Omega)$ can be controlled in terms of the measure and curvature of $\partial\Omega$ (see Section~\ref{sec_bo}).
In other contexts, such as discrete or non-convex spaces, $\ii(\Omega)$ cannot be related to the topological boundary of $\Omega$. As we shall see, the inflation constant is compatible with natural approximation procedures, and that is its main merit for the problem under study.

\noindent $\bullet$ The doubling condition \ref{doubling} is easy to check in many examples and allows us prove stronger estimates. 

\subsection{Poincar\'e perimeter}
We introduce the following notion of perimeter.
\begin{definition}
    Let $\rho:X\to [0,\infty)$ be a Borel measurable function and  $u\in L^1_{\text{loc}}(X)$. We say that $(u,\rho)$ are a \emph{Poincar\'e pair} if there exists $\lambda\ge 1$ such that for every ball $B_r(x)$ with $\mu(B_r(x)) >0$,
    \[\fint_{B_r(x)}|u-u_{B_r(x)}| \, d\mu \le r\fint_{B_{\lambda r}(x)} \rho \, d\mu\, ,\]
    where $f_E=\fint_E f d\mu=\tfrac{1}{\mu(E)}\int_E f d\mu$.
  
    We define the \emph{Poincar\'e perimeter} of a measurable set $\Omega\subseteq X$ as
\begin{align}\label{eq_per}    \pp(\Omega)=\inf\Big\{\liminf_{j\to\infty}\int_X \rho_{j} \,d\mu : \ (u_j,\rho_{j}) \text{ Poincar\'e pair, } u_j \text{ locally Lipschitz, } u_j\xrightarrow{L^1_{\text{loc}}(X)} 1_\Omega\Big\}.
\end{align}
\end{definition}
The definition of Poincar\'e perimeter is non-standard, but inspired by similar notions in metric spaces related to the total variation of $1_\Omega$ (see, for example, \cite{Mi}). While other more common definitions assign perimeter $0$ to every set if $X$ is discrete, as we shall see, our variant allows us to treat discrete metric spaces and is compatible with discretization and approximation procedures.   For a large family of metric spaces, our notion of perimeter is comparable to more standard ones, see Section~\ref{sec:perimeter}.

\subsection{Off-diagonal decay of the reproducing kernel}
To quantify the off-diagonal decay of the reproducing kernel $K$, for $s\ge 0$, we define
the \emph{dyadic decay measure}
\begin{align}\label{eqn}
    N(s) &:=    \sum_{n\ge 0}\sup_{x\in X} \int_{A_{n,x}} (1+ \mathrm{d}(x,x'))^{s} \|K(x,x')\|_{\S_2}^2 \, d\mu(x')\, ,
\end{align}
where $A_{n,x}=B_{2^n}(x)\smallsetminus B_{2^{n-1}}(x)$ for $n\ge 1$ and $A_{0,x}=B_{1}(x)$. 

If $X$ has a group structure that leaves the measure $\mu$ and the distance $\mathrm{d}$ invariant, and acts isometrically on $\hh$, then $N(s)$ simplifies to
\begin{align*}
N(s)=  \int_{X} (1+\mathrm{d}(e,x))^{s} \|K(e,x)\|_{\S_2}^2 \, d\mu(x)\,,
\end{align*}
where $e$ is the neutral element; see Section~\ref{sec_simod}.

\subsection{Main result}
With the notation $\log^\ast(x):=\max\{1,\log x\}$, our main result reads as follows.
\begin{theorem} \label{th:general_bound}
    Let $\delta\in(0,1)$ and $\tau \coloneqq \max\{\frac{1}{\delta},\frac{1}{1-\delta}\}$. Assume Conditions~\ref{norm2} and \ref{stripe} hold. Then
    \begin{align}  
        \big| \#\{\lambda  &\in \sigma(T_{\Omega}):\  \lambda > \delta\} -  \mu(\Omega) \big|
          \lesssim \ii(\Omega)  \cdot\inf \Big\{\big(\tau  N(s)\big)^{\frac{\gamma}{s}}\Big(\log^\ast\Big(\big(\tau  N(s)\big)^{\tfrac{1}{s}}\Big)\Big)^{1-\tfrac{\gamma}{s} } :   s\ge\gamma\Big\} \, . \label{eq:last_common}
      \end{align}
If in addition Condition~\ref{doubling} holds, 
then one also has
\begin{align}
        \big|  \#\{\lambda &\in \sigma(T_{\Omega}):\  \lambda > \delta\} - {   } \mu(\Omega) \big| \notag
       \\
       &\lesssim \max\{ \ii(\Omega),\pp(\Omega)\} \cdot  \inf \Big\{\big(\tau  { }N(s)\big)^{\frac{\gamma}{s+\gamma-1}}\Big(\log^\ast\Big(\big(\tau  {}N(s)\big)^{\tfrac{1}{s+\gamma-1}}\Big)\Big)^{\tfrac{s-1}{s+\gamma-1} } :  s\ge 1\Big\} \, . \label{eq:last_common2}
    \end{align} 
Here, the constant implied in \eqref{eq:last_common} is absolute, while the one in \eqref{eq:last_common2} is $O(C_X^4)$.
\end{theorem}

For kernels with exponential off-diagonal decay, Theorem~\ref{th:general_bound} takes the following form.
\begin{corollary}[Exponential decay] \label{lemma:exp}
Assume Conditions~\ref{norm2} and \ref{stripe} hold. Let $\alpha,\beta>0$ and suppose that  
\begin{align} \label{eq:exp}
D_\hh=\sup_{x'\in X} \int_X e^{\alpha \mathrm{d}(x,x')^{1/\beta}} \|K(x,x')\|_{\S_2}^2\, d\mu(x) <\infty \,.
\end{align}
Then, for $\delta\in(0,1)$ and $\tau \coloneqq \max\{\frac{1}{\delta},\frac{1}{1-\delta}\}$,
\begin{align} 
        \label{eq:last_common3}
        \big| \#\{\lambda \in \sigma(T_{\Omega}): \lambda > \delta\} -  {  }\mu(\Omega) \big|
        &\lesssim  \ii(\Omega) \cdot \big(\log^\ast(\tau {D_\hh} )\big)^{ {\beta}\gamma} \cdot\log^\ast\big(\log^\ast(\tau {D_\hh}) \big)  \,. 
\end{align} 
$($The implied constant depends on $\alpha$ and $\beta$.$)$
\end{corollary}

\subsubsection{Applicability}
In Section~\ref{sec_bo} we develop estimates on $\ii(\Omega)$ and $\pp(\Omega)$ that facilitate the application of Theorem~\ref{th:general_bound} and Corollary~\ref{lemma:exp}. Notably, if the measure $\mu$ is doubling and $X$ is a Poincar\'e space, then 
\begin{align*}
\pp(\Omega)\lesssim \HH(\partial\Omega)\,,
\end{align*}
where $\HH(\partial\Omega)$ is the 
\emph{codimension-one Hausdorff measure} of $\partial\Omega$ --- see Proposition~\ref{peri}.
In addition, if $\partial\Omega$ is lower Ahlfors regular
\cite{DaSe} and $X$ is quasiconvex, we show in Proposition~\ref{reduct} that
\begin{align}
\ii(\Omega) \lesssim \HH(\partial\Omega)\,,
\end{align}
where the implied constant shall be described precisely. 

Hence, in many situations of interest, the error rate of the spectral deviation estimates \eqref{eq:last_common}, \eqref{eq:last_common2} and \eqref{eq:last_common3} corresponds to the usual perimeter of the localization domain. The importance of the more precise estimates, formulated in terms of $\pp(\Omega)$ and $\ii(\Omega)$, is that
they are compatible with various 
\emph{discretization methods} --- see Section~\ref{sec_stab}. Thus, our work enables the rigorous analysis of \emph{approximation schemes} of Toeplitz operators, 
helping to show that the discrete implementations of Toeplitz operators share the spectral profile of their continuous counterparts.

\subsection{Applications}
\subsubsection{Gabor multipliers}
Let us state how Theorem~\ref{th:general_bound} applies to Gabor multipliers (see Section~\ref{sec:application} for definitions). 
\begin{corollary}
\label{cor_gm_prel}
Let $g \in L^2(\mathbb{R})$ with $\|g\|_2=1$, $N\in \N$ a resolution parameter and $\Omega\subset\R^{2}$ a compact domain with connected boundary and $\HH(\partial\Omega)\ge 1$. Suppose that the Gabor system $\big\{ g(t-\tfrac{n}{N}) e^{ 2\pi i \tfrac{ m }{N}t}\big\}_{n,m\in\Z}$ is a frame of $L^2(\mathbb{R})$ and consider the Gabor multiplier $M_{g,N,\Omega}$ from \eqref{eq_intro_gm}. 
The following statements hold:
\begin{enumerate}[label=(\roman*)] 
\item\label{cgi2} Suppose that $g\in M_{s}^1(\R)$ for some $s\ge \tfrac{1}{2}$ (see \eqref{eqmod}). There exist  $N_0= N_0(g,s)\ge 1$, and $C=C(s)>0$ such that if $N \ge N_0$, then 
\begin{align*}
        \big| \#\{\lambda \in \sigma(M_{g,N,\Omega}):\  \lambda > \delta\} -  |\Omega| \big| \notag
       \le 
         C\cdot\HH(\partial \Omega)
         \cdot  \big(\tau \|g\|_{M_{s}^1}^4\big)^{\frac{2}{2s+1}}\cdot\Big(\log^\ast\Big((\tau \|g\|_{M_{s}^1}^4)^{\tfrac{1}{2s+1}}\Big)\Big)^{\tfrac{2s-1}{2s+1} }\,.
    \end{align*}
\item\label{cgii2} Suppose that $g$ satisfies the Gelfand-Shilov condition $|V_gg(z)| \leq B e^{-\alpha |z|^{1/\beta}}$
for constants $B,\alpha>0$ and $\beta\ge 1$. There exist $N_0= N_0(B,\alpha,\beta)\ge 1$ and $C=C(B,\alpha,\beta)>0$ such that if $N \ge N_0$, then 
\begin{align*}
        \big|  \#\{\lambda \in \sigma(M_{g,N,\Omega}):\  \lambda > \delta\} - {   } |\Omega| \big| \notag
       \le C \cdot \HH(\partial \Omega) \cdot  \big(\log^\ast(B\tau )\big)^{ 2{\beta}} \cdot\log^\ast\big(\log^\ast(B\tau) \big)\,.
    \end{align*}
\end{enumerate}  
\end{corollary}
The strength of these estimates lies in the fact that \emph{the right-hand sides do not depend on any discretization parameter} (in particular, they do not involve the so-called dual Gabor window). The frame assumption of the Corollary is not a restriction, as it holds automatically for all large $N$. A version of Corollary \ref{cor_gm_prel} holds in higher dimension and for time-frequency parameters discretized along a general full-rank lattice $\Lambda$; in this setting, the eccentricity of $\Lambda$ and the regularity of $\partial\Omega$ come into play (see Theorem~\ref{cor:gabor}).

\subsubsection{Spectral stability under discretization with frames}

More generally, Theorem \ref{th:general_bound} applies to frames $\{\varphi_\lambda\}_{\lambda\in\Lambda}$ over a full-rank lattice $\Lambda$, with an analogous definition of frame multiplier associated to a compact domain $\Omega$ with lower Ahlfors regular boundary (see Section \ref{sec_frames}). The dependence on the frame is present in two ways: one has to control the spread of the diagonal entries and the off-diagonal decay of the cross-Gram matrix $\langle  \ddual_{\lambda'}, \varphi_\lambda \rangle$, where $\{\ddual_\lambda\}_{\lambda\in\Lambda}$ is the canonical dual frame. While the former can be estimated by the ratio of the frame bounds, the latter requires an explicit control in terms of the dual frame (an issue that we overcome in the time-frequency context resorting to the core of Gabor theory).

\subsubsection{The Fourier transform and slow decaying kernels}
The two-parameter spectral deviation estimates for time-frequency localization operators discussed in Section~\ref{sec_intro_tf} parallel current developments in the study of concentration operators of the Fourier transform. Here, the relevant function space is the space of bandlimited functions, that is, square-integrable functions with Fourier transform supported on a given compact interval of the real line. One-parameter spectral asymptotics go back to Landau and Widom \cite{La2,LaWi} and are a building block of
signal processing applications \cite{SlPo,LaPo2,LaPo3,Th,ZeMe}. Motivated by modern applications, two-parameter estimates for Fourier concentration operators have been recently developed in \cite{KaRoDa,Os,israel15,BoJaKa,DaWa,MR4711850,kul2}, while the case of higher dimensional general 
domains was treated in \cite{IsMa,eigenfourier,HuIsMa,HuIsMa2,kulikov2026sharp}. 

Theorem~\ref{th:general_bound} does not apply directly to Fourier concentration operators because the reproducing kernel of the space of bandlimited functions, the sinc kernel, is not integrable. The existing spectral deviation results for Fourier concentration operators use Fourier-specific techniques, such as, wave-packet expansions \cite{israel15,IsMa,eigenfourier,HuIsMa,HuIsMa2}. On the other hand, it is possible to interpret such methods as a decomposition of the space of bandlimited functions into a direct sum of spaces with fast-decaying reproducing kernels, which do fall in the scope of Theorem~\ref{th:general_bound}. We present the details in Section~\ref{sec:fou}, thus illustrating how Theorem~\ref{th:general_bound} can be applied (after some work) to certain slow-decaying kernels.

\subsubsection{Quantum harmonic analysis} Time-frequency localization has also been studied in vector-valued contexts \cite{accumulated-cohen}, often under the name of quantum harmonic analysis. In Section~\ref{sec_mixed} we show how Theorem~\ref{th:general_bound} applies to
so-called \emph{mixed-state localization operators} greatly improving the results of \cite[Theorem 4.4 and Lemma 5.3]{accumulated-cohen} in terms of the dependence on the spectral parameter.

\subsection{Organization} 
In Section~\ref{prel}, we set up notation and provide some background on vector-valued RKHS. Sections~\ref{sec_bo} and \ref{sec_disper} develop the relevant notions of (discrete) perimeter and boundary regularity.
In Section~\ref{sec:hankel} we prove estimates on the Schatten-$p$ (quasi) norms of the Hankel operator; these estimates provide the technical input for Section~\ref{subsec:deviation}, where we prove the main results of the paper. Section~\ref{sec_frames} and \ref{sec:application} are then devoted to applications to frame multipliers and Gabor multipliers.
Finally, in Section~\ref{sec_moreappl}, we provide further examples and applications. In particular, we discuss mixed state localization operators and revisit the argument in \cite{israel15} for Fourier concentration operators from our point of view.

\section{Preliminaries}
\label{prel}

\subsection{Notation}
We write $a\lesssim b$ whenever $a \le C b$ for some constant $C>0$ possibly depending on the parameters $\alpha,\beta,B$ (from  \eqref{eq:exp}, \eqref{eq_gsc2}), but not on any others. Similarly, we write $a\simeq b$ whenever $a\lesssim b \lesssim a$. Also, we will make use of the notation $\log^\ast(x):=\max\{1,\log x\}.$ 
When we need to stress the dependence of a quantity on a certain parameter, we add additional subscripts. For example, we may write \eqref{eqn} as $N_\hh(s)$.

\subsection{Vector-valued reproducing kernel Hilbert spaces} \label{sec_vvrkhs}

In this section, we  recall the definition and some basic properties of \emph{vector-valued reproducing kernel Hilbert spaces (RKHS)}. For a thorough introduction to that matter we refer to \cite[Chapter~6]{rkhs}.
Let $(X,\mu)$ be a measure space and $\HH$ be a separable Hilbert space. We consider a linear space $\mathbb{H} \subset \HH^X$ of functions from $X$ to $\HH$ with the following properties:
\begin{itemize}
\item[(a)] Each $f \in \hh$ is square integrable, i.e., $\int_X \|f(x)\|^2_{\HH}\, d\mu(x)<\infty$.
\item[(b)] If $f,g \in \hh$ are almost everywhere equal, i.e., $\mu\big(\{x\in X: f(x)\not=g(x)\}\big)=0$, then $f(x)=g(x)$ for every $x \in X$.
\item[(c)] With the embedding $\mathbb{H}\subset L^2(X,\HH)$ --- allowed by (a) and (b) --- the evaluation map $E_x:\mathbb{H}\to \HH$, $E_xf:= f(x)$ is continuous for every $x \in X$, i.e., $\|E_xf\|\leq C_x \|f\|_{L^2(X,\HH)}$, for some constant $C_x$. 
\end{itemize}
Such a space is called a \emph{vector-valued RKHS}. The \emph{reproducing kernel} $K:X\times X\to B(\HH)$
is defined as $K(x,y)=E_xE_y^\ast$, and satisfies $K(x,y)=K(y,x)^\ast$ and
$$
f(x)=\int_X K(x,y)f(y)\,d\mu(y)\, ,\qquad f\in \mathbb{H}\, .
$$
In the most common examples of RKHS, $\hh$ is a subspace of complex-valued functions, that is, $\HH=\C$.

\subsection{Trace of \texorpdfstring{{\boldmath ${T_\Omega}$}}{T Omega}}
We start by noting that the trace of $T_\Omega$ can be computed exactly as in the scalar-valued setting.
\begin{lemma}\label{lem:trace}
$T_\Omega$ is trace class and its trace is
    $$
    \emph{trace}(T_\Omega)=\mu(\Omega)\,.
    $$
\end{lemma}
\begin{proof}
  Let $\{g_n\}_{n \in J}\subset \HH$ and $\{\psi_k\}_{k\in I}\subset L^2(X)$ be two  orthonormal bases. Notice that $I$ and $J$ are at most countable. Then $\{\psi_k g_n\}_{(k,n)\in I\times J}$ is an orthonormal basis for $L^2(X,\HH)$. Since $T$ is a positive operator, it suffices to compute 
    \begin{align*}
        \text{trace}(T_\Omega)&=\sum_{(k,n)\in I\times J}\langle T_\Omega \psi_k g_n,\psi_k g_n\rangle_{L^2(X,\HH)}
        \\ &=\sum_{(k,n)\in I\times J}\int_X\int_\Omega\int_X  \big\langle K(x,x')K(x',y)\psi_k(y)g_n,\psi_k(x)g_n\big\rangle \,d\mu(y)d\mu(x')d\mu(x)  
         \\ &=\sum_{k\in I}\int_X\int_\Omega\int_X \psi_k(y)\overline{\psi_k(x)} \big\langle  K(x',y)  ,K(x',x)\big\rangle_{\S_2} \, d\mu(y) d\mu(x') d\mu(x)
            \\ &=\int_\Omega \int_X \sum_{k\in I} \Big\langle    \big\langle  K(x',y)  ,K(x',\cdot)\big\rangle_{\S_2}  ,\psi_k \Big\rangle_{L^2(X)} \psi_k(y) \, d\mu(y) d\mu(x')
            \\ &=\int_\Omega \int_X \| K(x',y)\|_{\S_2}^2 \,d\mu(y) d\mu(x')=\mu(\Omega)\,,
    \end{align*}
    where in the last step we used \ref{norm2}.
\end{proof}

\subsection{Simplified off-diagonal decay}\label{sec_simod}
If $X$ has a group structure such that $\mu(E)=\mu(gE)$, 
for every Borel set $E$,
$\mathrm{d}(x,x')=\mathrm{d}(gx,gx')$ and $K(x,x')=K(gx,gx')$  and every $g,x,x'\in X$, then the measure of off-diagonal decay of the reproducing kernel \eqref{eqn} simplifies to
\begin{align}
N(s)=  \int_{X} (1+\mathrm{d}(e,x))^{s} \|K(e,x)\|_{\S_2}^2 \, d\mu(x)\,,
\end{align}
where $e$ is the neutral element. Indeed,
with the notation
$A_{n,x}=B_{2^n}(x)\smallsetminus B_{2^{n-1}}(x)$ for $n\ge 1$ and $A_{0,x}=B_{1}(x)$,
\begin{align*}
    \int_{A_{n,x}} (1+ \mathrm{d}(x,x'))^{s} \|K(x,x')\|_{\S_2}^2 \, d\mu(x')
    &= \int_{x^{-1}A_{n,x}} (1+ \mathrm{d}(x,xx'))^{s} \|K(x,xx')\|_{\S_2}^2 \, d\mu(x')
    \\&= \int_{A_{n,e}} (1+ \mathrm{d}(e,x'))^{s} \|K(e,x')\|_{\S_2}^2 \, d\mu(x')\,.
\end{align*}

\section{Boundary and perimeter}\label{sec_bo}

\subsection{Ahlfors regularity and codimension one measure}
\begin{definition}
For a set $E\subseteq X$, and $r>0$ define
\[\HH_r(E)=\inf\Big\{\sum_j \tfrac{\mu(B_{r_j}(x_j))}{r_j}:\ E\subseteq\bigcup_{j} B_{r_j}(x_j),\  r_j\le r\Big\}\,,\]
where $j$ runs over an at most countable index set. The \emph{codimension one Hausdorff measure} is
\[\HH(E)=\lim_{r\to 0} \HH_r(E)\,.\]
\end{definition}
The codimension one measure of the boundary of a set, $\HH(\partial E)$, provides a natural notion of perimeter, which we 
call \emph{geometric perimeter}. We shall compare it to \eqref{eq_per} and relate it to the inflation constant. For the moment, we use the codimension one measure to introduce the following notion of regularity.
\begin{definition}\label{defreg}
    We say a Borel set $E\subseteq X$ is regular at scale $\eta>0$ if there is a constant $\kappa>0$, such that
    \[\HH(E\cap B_r(x))\ge \frac{\kappa}{r}\, {\mu(B_r(x))},\qquad 0<r\le \eta, \qquad x\in E \,.\]
\end{definition}
This definition essentially corresponds to the notion of (lower) \emph{Ahlfors regularity} (see \cite[Definition I.1.13]{DaSe}). When applied to the boundary of a set $E=\partial \Omega$, the condition can be interpreted as a curvature bound.

\subsection{Perimeter in Poincar\'e spaces}\label{sec:perimeter}
In this section we consider so-called \emph{doubling Poincar\'e spaces} and show that, in that setting, the notion of perimeter introduced in \eqref{eq_per} can be estimated in terms of the codimension one Hausdorff measure of the topological boundary.

We start with a few definitions. A curve is a continuous map $\sigma: [a,b] \to X$ with $a,b\in\R$. Its length is defined as
    \[\ell(\sigma)=\sup \sum_{j=1}^n \mathrm{d}(\sigma(t_j),\sigma(t_{j+1}))\,,\]
    where the supremum is taken over all finite partitions $a=t_1\le t_2 \le \ldots \le t_n\le t_{n+1}=b$. A curve of finite length is called \emph{rectifiable}. 
    For a nonnegative Borel function $\rho:X\to [0,\infty)$ we define:
    \[\int_\sigma \rho = \int_0^{\ell(\sigma)}\rho\circ\widetilde\sigma(t)\, dt\,,\]
    where $\widetilde\sigma$ is the unique arc length reparametrization of $\sigma$. 
    We say $\rho$ is an \emph{upper gradient} (originally introduced as a \emph{very weak gradient} in \cite{HK}) of a real-valued function $u$ on $X$ if
    \[|u(x)-u(y)|\le \int_{\sigma_{xy}} \rho\,,\]
    for every rectifiable curve $\sigma_{xy}$ joining $x$ and $y$.

    The \emph{total variation} of $u\in L^1_{\text{loc}}(X)$ is 
    \[\|Du\|=\inf\Big\{\liminf_{j\to\infty}\int_X \rho_{j} \,d\mu : \ \rho_{j} \text{ upper gradient of } u_j, u_j \text{ locally Lipschitz, } u_j\xrightarrow{L^1_{\text{loc}}(X)} u\Big\}\,.\]
    In this setting,   $\|D1_E\|$ provides another notion of perimeter for a measurable set $E\subseteq X$. Finally, a metric space $X$ is called a \emph{1-Poincar\'e space} (Poincar\'e space for short) if there are constants $C_P,\lambda>0$ such that for every ball $B_r(x)$, every locally integrable function $u$ on $X$, and every upper gradient $\rho$ of u, we have
    \[\fint_{B_r(x)}|u-u_{B_r(x)}| \,d\mu\le C_P r\fint_{B_{\lambda r}(x)}\rho \,d\mu,\]
    where $u_{B_r(x)}=\fint_{B_r(x)}u \,d\mu\,.$

Some examples of Poincar\'e spaces include $(\R^n,\omega(x) dx)$ with $\omega$ in the Muckenhoupt weight class $\mathcal{A}_1$,  complete Riemannian manifolds with non-negative Ricci curvature and the Heisenberg group (see \cite[Appendix A]{BB} and the references therein for a more comprehensive list).
We note the following very useful estimate from \cite{AmMiPa}.

\begin{proposition}[Geometric perimeter controls the  Poincar\'e perimeter]\label{peri}
Let $X$ be a doubling  Poincar\'e  space   with constant $C_X$. Then, for every measurable set $E$,
\[
\pp(E)\le C_P C_X  \HH(\partial E)\,.
\]
\end{proposition}
\begin{proof}
Since $(u,C_P\rho)$ is a Poincar\'e pair whenever $\rho$ is an upper gradient of $u$, we have
    \[\pp(E)\le C_P \|D1_E\|\,.\]
In addition, by \cite[Theorem 4.6]{AmMiPa}, \begin{align}\label{eq_diex}
    \|D1_E\|\le C_X \HH(\partial E)\,.
    \end{align}
    In fact, \eqref{eq_diex} holds with the potentially smaller measure theoretical boundary $\partial^*E$ on the right-hand side (see also \cite{La} for more refined estimates).
\end{proof}
While Poincar\'e spaces are connected (see, for example, \cite[Proposition 4.2]{BB}),
some of the applications that motivate us involve disconnected spaces, and that is why we introduced the 
more general notion of perimeter $\pp(E)$, cf. \eqref{eq_per}. 

\subsection{Quasiconvexity}
We now look into Condition~\ref{stripe}. If this condition holds, one may expect the corresponding constant $\ii(\Omega)$ to be related to
the size of the boundary of $\Omega$. While this intuition may fail in general --- since the distance to a set may not be attained at its boundary --- it is correct in the Euclidean space, and, as we shall see, in the following context.

\begin{definition}\label{defconvex}
    We say that a metric space $X$ is \emph{$M$-quasiconvex}, where $M\ge 1$, if every pair of points $x,y\in X$ can be joined by a curve of length $\le M \mathrm{d}(x,y)$.
\end{definition}

It is easy to check that if $X$ is $M$-quasiconvex, then for $E\subseteq X$ with nonempty boundary, for every $x\in E^c$,
\[\mathrm{d}(x,\partial E)\le M \mathrm{d}(x,E)\,.\] 
Normed spaces are of course $1$-quasiconvex. Furthermore,
complete doubling Poincar\'e spaces are quasiconvex \cite[Theorem 4.32]{BB}. 

We now show that \ref{doubling} implies \ref{stripe} for quasiconvex spaces, and that under additional Ahlfors regularity, the geometric perimeter controls the inflation constant.

\begin{proposition}[Control of inflation constant by geometric perimeter]
    \label{reduct}
    Let $X$ be a doubling $M$-quasiconvex space and consider $\gamma=\log_2(C_X)$ its doubling dimension. For a compact set $\Omega\subseteq X$ let $E$ denote either $\Omega$ or $\Omega^c$. Then the following hold for every $R>0$:
     \begin{align}
        \label{eqred1}
        \mu\big(\{x\in E^c : \ \mathrm{d}(x,E)\le R\}\big) &\le (8M)^\gamma \max\{R^\gamma,1\} \mu\big(\{x\in X : \ \mathrm{d}(x,\partial \Omega)\le 1\}\big)\,.
    \end{align}
    Additionally, if $\partial \Omega$ is regular at scale $\eta>0$ with constant $\kappa>0$, then
    \begin{align}
        \label{eqred2}
        \mu\big(\{x\in E^c : \ \mathrm{d}(x,E)\le R\}\big) 
        &\le  \frac{(8M)^\gamma}{\kappa} R\big(1+\big(\tfrac{R}{\eta}\big)^{\gamma-1}\big) \HH(\partial\Omega)\,.
    \end{align}
(Here, we interpret $\mathrm{d}(x,\emptyset)=\infty$.)
\end{proposition}
It is worth mentioning that in the context of \eqref{eqred2} and assuming $X$ has more than one point, $\gamma\ge 1$ (see \cite{SoTr}). For $X=\R^d$, \eqref{eqred2} also follows from \cite{Ca}.

\begin{proof}[Proof of Proposition~\ref{reduct}]
    Note that $\mathrm{d}(x,\partial \Omega)\le M \mathrm{d}(x,E)$ for every $x\in E^c$,
    and so,
    \[\mu(\{x\in E^c : \ \mathrm{d}(x,E)\le R\})\le \mu(\{x\in X : \ \mathrm{d}(x,\partial \Omega)\le M R\})\,.\]
    We let $r>0$ and cover $\partial \Omega$ with all the balls of radius $r$ centered at points of that set and extract a Vitali sub-family, that is, a collection of disjoint balls $\{B_{r}(x_j)\}_{j=1}^J$, with $x_j \in \partial\Omega$, 
    such that $\partial\Omega\subseteq \bigcup_{j=1}^J B_{3r}(x_j)$. We start by choosing $r=1$, which gives
    \[\sum_{j=1}^J\mu( B_{1}(x_j))=\mu\Big(\bigcup_{j=1}^J B_{1}(x_j)\Big)\le \mu(\{x\in X : \ \mathrm{d}(x,\partial\Omega)\le 1\})\,.\]
    In addition, if $\mathrm{d}(x,\partial\Omega)$ is attained at $y\in\partial\Omega$, there exists $1\le j\le J$ such that
    \[\mathrm{d}(x,x_j)\le \mathrm{d}(x,y)+\mathrm{d}(y,x_j)\le \mathrm{d}(x,\partial\Omega)+3 \,.\]
    As a consequence, if $2^k\le 3+MR<2^{k+1}$,    
    \begin{align*}
        \mu(\{x\in E^c : \ \mathrm{d}(x,E)\le R\}) &\le \mu(\{x\in X : \ \mathrm{d}(x,\partial \Omega)\le M R\})
        \\ & \le \sum_{j=1}^J\mu( B_{2^{k+1}}(x_j)) \le 2^{\gamma(k+1)} \sum_{j=1}^J\mu( B_{1}(x_j))
        \\
        &\leq 2^\gamma(3+M R)^\gamma\sum_{j=1}^J\mu( B_{1}(x_j))
        \\ & \le (8M )^\gamma \max\{R^\gamma,1\} \mu(\{x\in X : \ \mathrm{d}(x,\partial\Omega)\le 1\})\,,
    \end{align*}
    where we used $\gamma=\log_2(C_X)$.
    This proves \eqref{eqred1}. Regarding \eqref{eqred2}, we
    choose $r=\min\{MR,\eta\}$, consider a Vitali sub-family of balls as before and repeat the previous argument to conclude that
    \begin{align*}
        \mu(\{x\in E^c : \ \mathrm{d}(x,E)\le R\}) & \le 2^\gamma (3+M R r^{-1})^\gamma \sum_{j=1}^J\mu( B_{r}(x_j))
        \\&\le  2^\gamma r(3+M R r^{-1})^\gamma \kappa^{-1}\HH(\partial\Omega) \notag
        \\&\le 8^\gamma r (M R r^{-1})^\gamma \kappa^{-1}\HH(\partial\Omega)
        \\&\le  8^\gamma \max\{MR,(MR)^\gamma \eta^{-\gamma+1}\} \kappa^{-1}\HH(\partial\Omega) \notag
        \,.
    \end{align*}
    We can assume without loss of generality that $X$ has more than one point, which implies that $C_X\ge 2$, or equivalently, $\gamma\ge1$ (see \cite{SoTr}). In particular, \[\max\{MR,(MR)^\gamma \eta^{-\gamma+1}\}\le M^\gamma R\big(1+\big(\tfrac{R}{\eta}\big)^{\gamma-1}\big)\,,\]
    and \eqref{eqred2} follows. 
\end{proof}

\subsection{Non-local perimeters}

We now study expressions of the form $\int_{\Omega} \int_{\Omega^c} \varphi(x,y)\,dx\,dy$, known as \emph{non-local perimeters}, in terms of $\Omega$ and the off-diagonal decay of $\varphi$. Estimates of this kind are well-known for $\R^d$, and usually stated for convolution kernels $\varphi(x-y)$ (e.g. \cite[Proposition 1.4]{MaRoTo}). We derive variants of such estimates for rather general spaces.

\begin{lemma}[First non-local perimeter estimate] \label{lem:p2}
Let $\Omega\subseteq X$ be a compact set and consider $\varphi: X\times X \to [0,\infty)$ measurable. Assuming that Condition~\ref{stripe} holds, we have
    \begin{align} \label{eq:p2}
        \int_{\Omega}\int_{\Omega^c} \varphi(x,y) \, d\mu(y)d\mu(x) 
      \lesssim  
            \ii(\Omega) \sum_{n\ge 0} \sup_{x\in X} \int_{A_{n,x}} (1+\mathrm{d}(x,y))^{\gamma}\varphi(x,y)  \, d\mu(y)\,,
    \end{align}
    where $A_{n,x}=B_{2^n}(x)\smallsetminus B_{2^{n-1}}(x)$ for $n\ge 1$ and $A_{0,x}=B_{1}(x)$.
\end{lemma}
\begin{proof}
For $n\ge 0$, define
\begin{align*}
 \Omega_n&=\{x \in \Omega: \  \mathrm{d}(x,\Omega^c)\le 2^{n}\}\,. 
\end{align*}
Notice that if $x\in \Omega$ and $y\in \Omega^c\cap A_{n,x}$, then 
\[ \mathrm{d}(x,\Omega^c)\le \mathrm{d}(x,y)\le 2^{n}\,,\]
and therefore,  $x\in \Omega_n$. So,
\begin{align*}
\int_{\Omega}\int_{\Omega^c} \varphi(x,y) \, d\mu(y)d\mu(x) 
& = \sum_{n\ge 0} \int_{\Omega}\int_{\Omega^c\cap A_{n,x}} \varphi(x,y) \, d\mu(y)d\mu(x) 
\\ & = \sum_{n\ge 0} \int_{\Omega_n}\int_{\Omega^c\cap A_{n,x}} \varphi(x,y) \, d\mu(y)d\mu(x) 
\\ &\le \sum_{n\ge 0}  2^{-\gamma(n-1)}  \mu(\Omega_n)\, \sup_{x\in X} \int_{A_{n,x}} (1+\mathrm{d}(x,y))^{\gamma}\varphi(x,y)  \, d\mu(y) 
\\ &\le \ii(\Omega) \sum_{n\ge 0} \sup_{x\in X} \int_{A_{n,x}} (1+\mathrm{d}(x,y))^{\gamma}\varphi(x,y)  \, d\mu(y) \,,
\end{align*}
where in the last step we used \ref{stripe}.
\end{proof}

\begin{lemma}[Second non-local perimeter estimate] \label{lem:p}
Let $\Omega\subseteq X$ be a compact set and consider $\varphi: X\times X \to [0,\infty)$ measurable. Assuming that Condition~\ref{doubling} holds, we have
    \begin{align} \label{eq:p}
        \int_{\Omega}\int_{\Omega^c} \varphi(x,y) \, d\mu(y)d\mu(x) 
        \lesssim  {C_X^4}
            \pp(\Omega) \sum_{n\ge 0} \begin{multlined}[t] \Big(\sup_{x\in X} \int_{A_{n,x}} (1+\mathrm{d}(x,y))\varphi(x,y)  \, d\mu(y)
            \\ + \sup_{y\in X} \int_{A_{n,y}} (1+\mathrm{d}(x,y))\varphi(x,y)  \, d\mu(x) \Big)\,,
        \end{multlined}
    \end{align}
    where $A_{n,x}=B_{2^n}(x)\smallsetminus B_{2^{n-1}}(x)$ for $n\ge 1$ and $A_{0,x}=B_{1}(x)$.
\end{lemma}
\begin{proof}
Let us show that we can assume that the right-hand side of \eqref{eq:p} is positive and finite without loss of generality. If it is infinite then the inequality holds trivially. Similarly, if it is zero then either $\varphi=0$ almost everywhere or $\pp(\Omega)=0$. If $\varphi=0$, again the inequality holds trivially. If $\pp(\Omega)=0$, then it follows from the definition that for every ball $B$, 
\[\fint_B |1_\Omega-(1_\Omega)_B|= 2\frac{\mu(\Omega\cap B)\mu(\Omega^c\cap B)}{\mu(B)^2}=0\,.\]
A straightfoward computation now shows that $\mu(\Omega)=0$ or $\Omega=X$ and in both cases \eqref{eq:p} trivially holds.

For $\varepsilon>0$ let $(u_j,\rho_j)$, $j\in\N$, be Poincar\'e pairs such that $u_j$ is locally Lipschitz, $u_j\to 1_\Omega$ in $L^1_{\text{loc}}(X)$ and 
\begin{align}
\label{eq_uppgrad}
    \lim_{j\to\infty}\int_X \rho_j \, d\mu\le \pp(\Omega)+\varepsilon\,.
\end{align}
We also assume that $0\le u_j \le 1$ for every $j\in \N$, by truncating them if necessary and replacing $\rho_j$ by $2\rho_j$ (note that his may add an extra factor of 2 in \eqref{eq_uppgrad}, but this does not affect the argument). 
Indeed, for any contractive $f:\R\to\R$ and any $B$ of finite measure,
\[\fint_B |f\circ u-(f\circ u)_B|\le \fint_B \fint_B |f\circ u(x)-f\circ u(y)|\le \fint_B \fint_B |u(x)- u(y)|\le 2  \fint_B | u- u_B|\, . \]
It is easy to check that under these assumptions,
\[\int_{\Omega}\int_{X} ((u_j(x)-u_j(y)) \varphi(x,y) \, d\mu(y)d\mu(x)\xrightarrow{j\to\infty}\int_{\Omega}\int_{\Omega^c} \varphi(x,y) \, d\mu(y)d\mu(x)\,.
\]
So, it remains to show that for a Poincar\'e pair $(u,\rho)$ with $u$ continuous,
\begin{align}\label{eqp0}
\int_{X}\int_{X} |u(x)-u(y)| \varphi(x,y) \, d\mu(y)d\mu(x)
\lesssim C_X^4  \int_{X} \rho \,d\mu \sum_{n\ge0} \sup_{x\in X} \int_{A_{n,x}} (1+\mathrm{d}(x,y))\varphi(x,y)  \, d\mu(y)\,.
\end{align}
Let us estimate $|u(x)-u(y)|$ by a chaining argument as in the proof of \cite[Lemma~5.15]{HK}. For $k\ge 0$, define $r_k=2^{1-k}(1+\mathrm{d}(x,y))$, $B_k=B_{r_{k+1}}(x)$ and $B_k'=B_{r_{k}}(y)$. By construction $B_0\subseteq B_0'\subseteq B_{2r_0}(x)$, $B_{k+1}\subseteq B_k$ and $B_{k+1}'\subseteq B_k'$.
From Condition~\ref{doubling},
\begin{align*}
  |u_{B_0}-u_{B_0'}|&\leq \fint_{B_0}|u(w)-u_{B_0'}|d\mu(w)  \leq \frac{\mu(B_0')}{\mu(B_0)} \fint_{B_0'}|u(w)-u_{B_0'}|d\mu(w)
  \\
  &\leq C_X^2  \fint_{B_0'}|u(w)-u_{B_0'}|d\mu(w)\,.
\end{align*}
Using similar arguments to bound $|u_{B_k}-u_{B_{k+1}}|$ and  $|u_{B_k'}-u_{B_{k+1}'}|$ then shows
\begin{align}\label{eqp2}
    |u(x)-u(y)|&\le |u_{B_0}-u_{B_0'}| + \sum_{k=0}^\infty (|u_{B_k}-u_{B_{k+1}}|+|u_{B_k'}-u_{B_{k+1}'}|)
    \\ &\begin{multlined}[t]
        \le C_X^2 \fint_{B_0'} |u(w)-u_{B_0'}|\, d\mu(w)\notag
     \\+ \sum_{k=0}^\infty \Big(\fint_{B_{k}}|u(w)-u_{B_{k}}|\, d\mu(w) +\fint_{B_{k}'}|u(w)-u_{B_{k}'}|\, d\mu(w)\Big)\notag
    \end{multlined}
    \\ & \lesssim C_X^2 (1+\mathrm{d}(x,y)) \sum_{k=0}^\infty 2^{-k}\Big(\fint_{B_{\lambda r_{k}}(x)}\rho(w)\, d\mu(w) + 
     \fint_{B_{\lambda r_k}(y)}\rho(w)\, d\mu(w)\Big)\,.\notag
\end{align}
In order to proceed, we need to deal with the fact that $r_k$ depends on $x,y$. Notice that if $y\in A_{n,x}$ for some $n\ge 0$, then 
\[2^{1-k}(1+2^{n-1})\le r_k\le 2^{1-k}(1+2^n)\,.\]
In this case, by the doubling assumption \ref{doubling}, for every $k\ge 0$,
\[\frac{1_{B_{\lambda r_k}(x)}(w)}{\mu(B_{\lambda r_k}(x))} \le C_X^2\frac{1_{B_{\lambda 2^{1-k}(1+2^n)}(w)}(x)}{\mu(B_{\lambda 2^{1-k}(1+2^n)}(w))}\,.\]
From this and letting $\psi(x,y)=\varphi(x,y)(1+\mathrm{d}(x,y))$, we see that
\begin{align}\label{eqp3}
\int_X \int_{A_{n,x}} \fint_{B_{\lambda r_{k}}(x)}&\rho(w)  \psi(x,y) \, d\mu(w)d\mu(y)d\mu(x) 
\\ &\le C_X^2 \int_X \int_{A_{n,x}} \int_{X} \frac{1_{B_{\lambda 2^{1-k}(1+2^n)}(w)}(x)}{\mu(B_{\lambda 2^{1-k}(1+2^n)}(w))} \rho(w)  \psi(x,y) \, d\mu(w)d\mu(y)d\mu(x) \notag
\\ &\le C_X^2 \sup_{x\in X} \int_{A_{n,x}}  \psi(x,y) \, d\mu(y) \int_X  \rho(w) \, d\mu(w) \,.\notag
\end{align}
Analogously,
\begin{align*}
\int_X \int_{A_{n,y}} \fint_{B_{\lambda r_{k}}(y)}&\rho(w)  \psi(x,y) \, d\mu(w)d\mu(x)d\mu(y) 
\le C_X^2 \sup_{y\in X} \int_{A_{n,y}}  \psi(x,y) \, d\mu(x) \int_X  \rho \, d\mu(w) \,.
\end{align*}
Joining this with \eqref{eqp2} and \eqref{eqp3} we get \eqref{eqp0}.
\end{proof}

\section{Discrete perimeter and grid approximations}\label{sec_disper}

\subsection{Perimeter in discrete grids}
One of our goals is to obtain spectral deviation estimates that are compatible with the discretization of $\R^d$ and corresponding approximation schemes. To investigate such approximations, let $\Lambda=A\Z^d\subseteq\R^d$, with $A\in \gl_d(\R)$, be a lattice and endow it with the Euclidean metric and normalized counting measure 
\begin{align}\label{eq_ncm}
\mu_\Lambda(\{\lambda\}) = |\Lambda|:=|\det A|.
\end{align}
Let also $0<\|A^{-1}\|^{-1}=\sigma_1(A)\le\ldots\le \sigma_d(A)=\|A\|$ be the singular values of $A$ and $\vr=\|A\|\|A^{-1}\|$ its \textit{condition number}. Note that if $\|A\|\to 0$, then
\begin{align}\label{eq_to}
d\mu_{\Lambda} \to dx
\mbox{ vaguely}.
\end{align}
We shall be interested in approximation schemes where, in addition to \eqref{eq_to}, the eccentricity of the grid remains bounded, as formalized in the following definition. 
\begin{definition}
The \emph{isotropic fineness measure} of a full-rank lattice $\Lambda = A \Z^d\subseteq\R^d$ is
\begin{align}\label{eq_isof}
\isof_\Lambda := \vr^{2d} \max\{1,\|A\|^{d}\}
= \|{A}\|^{2d} \|A^{-1}\|^{2d}
\max\{1,\|A\|^{d}\},
\end{align}
where $A \in \gl_d(\R)$.
\end{definition}
\begin{remark}[Stability under isotropic contractions]\label{rem_isof}
For a full-rank lattice $\Lambda= A \Z^d$, 
\[  \isof_{\varepsilon\Lambda} = {\color{purple}}
\isof_{\|A\|^{-1}\Lambda}=\vr^{2d}, \qquad 0<\varepsilon\le \|A\|^{-1}.\]
\end{remark}
More generally, we will consider  measures of the form
\begin{equation}\label{eq:weighted_measure}
    \mu_\omega:= \sum_{\lambda\in\Lambda}\omega_\lambda
 \delta_\lambda
\end{equation}
for some positive weights $\omega_\lambda>0$, and define $\vartheta_\omega:=\inf_{\lambda\in\Lambda }\omega_\lambda$ and $\Theta_\omega:=\sup_{\lambda\in\Lambda }\omega_\lambda$. Notice that
\begin{equation}\label{eq:comp-meas}
\vartheta_\omega \,\mu_\Lambda\leq |\Lambda| \,\mu_\omega\leq \Theta_\omega \,\mu_\Lambda.
\end{equation}
Motivated by \eqref{eq_to}, we consider the following notion of lattice boundary.
\begin{definition}[Boundary with respect to a lattice]
Let $\Omega \subseteq A \mathbb{Z}^d$ with $A\in \gl_d(\R)$ be finite. The \emph{discrete boundary} of $\Omega$ associated to $A$ is  \[\partial_A\Omega=\{x\in \Omega:\ \mathrm{d}(A^{-1}x,A^{-1}\Omega^c)=1\}=A\partial_{\id}A^{-1}\Omega \,,\]
where $\mathrm{d}(x,y)$ is the Euclidean distance.
\end{definition}
\begin{remark}[Thickening a subset of a lattice]\label{rem_th}
If $\Omega \subset \Lambda = A \mathbb{Z}^d$, then the perimeter of the set $\widetilde\Omega = \bigcup_{x\in\Omega} A[-\tfrac{1}{2},\tfrac{1}{2}]^d+x\subseteq \R^d$  satisfies
\begin{align}\label{eq_dlb}
\|A\|^{-1}\mu_{\Lambda}(\partial_{A}\Omega) \lesssim \mathcal{H}(\partial \widetilde\Omega) \lesssim \vr \,\|A\|^{-1}\mu_{\Lambda}(\partial_{A}\Omega)\,,
\end{align}
where $\mathcal{H}$ is the $(d-1)$-dimensional Hausdorff measure.

Indeed, from the area formula it follows that
\[\HH\big(\partial (A[-\tfrac{1}{2},\tfrac{1}{2}]^d)\big)=2\sum_{j=1}^d\sqrt{\det(A(j)^tA(j))}\,,\]
where $A(j)\in \R^{(d-1)\times d}$ is obtained by deleting the $j$-th column of $A$. By Cauchy-Binet,
\[\det(A(j)^tA(j))=\sum_{i=1}^d C_{ij}^2,\]
where $C$ is the cofactor matrix of $A$.
Since $C=\det(A) A^{-t}$, \eqref{eq_dlb} now follows from a straightforward computation using that
\[\HH\big(\partial (A[-\tfrac{1}{2},\tfrac{1}{2}]^d)\big)=2\sum_{j=1}^d\sqrt{\sum_{i=1}^d C_{ij}^2}\simeq \|C\|_F=\sqrt{\sum_{j=1}^d \prod_{i\neq j}\sigma_i(A)^2}\,,\]
 where $\|\cdot\|_F$ is the Frobenius norm.
\end{remark}
According to \eqref{eq_dlb}, 
\begin{align}\label{eq_dlb2}
{\|A\|}^{-1}\mu_{\Lambda}(\partial_{A}\Omega)
\end{align}
is a reasonable notion of lattice perimeter, which is compatible with regular approximation by grids --- see also Section~\ref{sec_stab}. We now compare \eqref{eq_dlb2} to $\pp(\Omega)$, as defined in \eqref{eq_per} with respect to the ambient space $\Lambda$. 

\begin{proposition}
    \label{lemdisper}
    Consider 
    $\Lambda=A\mathbb{Z}^d$ with $A\in \gl_d(\R)$ endowed with the Euclidean distance 
    $\mathrm{d}$ and a weighted measure $\mu_\omega$, cf. \eqref{eq:weighted_measure}. Then, for every finite set $\Omega\subseteq \Lambda$:
    \[
    {\left(\frac{\vartheta_\omega}{\Theta_\omega}\right)^5}
    \|A\|^{-1}\mu_\omega(\partial_A\Omega)\lesssim \pp(\Omega)\lesssim \left(\frac{\Theta_\omega}{\vartheta_\omega}\right)^3 \vr^{2d}\|A\|^{-1}\mu_\omega(\partial_A\Omega)
    \,,\]
    where the implied constants depend only on the dimension $d$. 
    In particular, for $\mu_\omega=\mu_\Lambda$,
    \[
    \|A\|^{-1}\mu_\Lambda(\partial_A\Omega)\lesssim \pp(\Omega)\lesssim \vr^{2d}\|A\|^{-1}\mu_\Lambda(\partial_A\Omega)
    \,.\]
\end{proposition}
To prove Proposition~\ref{lemdisper}, we will need the following particular case of \cite[Theorem 3.2]{BoGo}.

\begin{theorem}[{\cite[Theorem 3.2]{BoGo}}]
    \label{theoiso}
    Let $\mu=\mu_1\times\ldots\times\mu_d$ be a product of probability measures. Then, for any measurable set $\Omega$,
    \[I(\mu(\Omega))\le \sqrt{2} \ee\left[\Big(\sum_{j=1}^d\vv_j(1_\Omega)\Big)^{1/2}\right]\,,\]
    where $I=F'\circ F^{-1}$ with $F$ the cumulative distribution of a standard normal distribution and $\vv_j$ is the variance with respect to the $j$-th variable.
\end{theorem}

It is easy to check that $I$ is concave, $I(0)=I(1)=0$ and $I(1/2)=1/\sqrt{2\pi}$. In particular for $0\le x\le 1$,
\begin{align}
    \label{eqI}
    x(1-x)\le \tfrac{1}{2}-\big|x-\tfrac{1}{2}\big|\le \tfrac{\sqrt{2\pi}}{2}I(x)\,.
\end{align}

\begin{proof}[Proof of Proposition~\ref{lemdisper}] 
If $\vartheta_\omega = 0$, the claim holds trivially; therefore, we can assume $\vartheta_\omega >0$.
By a rescaling argument, we can assume without loss of generality that $\|A\|=1$. Indeed, given $\uptau>0$, $(u,\rho)$ is a Poincar\'e pair for the space $(\Lambda,\mu_\omega)$ if and only if  $(u(\tfrac{\cdot}{\uptau}),\tfrac{1}{\uptau}\rho(\tfrac{\cdot}{\uptau}))$ is a pair for $(\uptau\Lambda,\mu_{\uptau\omega})$. From this, the definition of perimeter and making explicit the dependence $\pp(E)=\pp_{\mu}(E)$ on the measure $\mu$, we see that for $\Omega\in\Lambda$,
\[\pp_{\mu_{\uptau\omega}}(\uptau\Omega)=\uptau^{d-1}\pp_{\mu_\omega}(\Omega)\,.\]
Also, we have that
\[\|\uptau A\|^{-1}\mu_{\uptau\omega}(\partial_{\uptau A}(\uptau\Omega))=(\uptau\|A\|)^{-1}\mu_{\uptau \omega}(\uptau\partial_A\Omega)=\uptau^{d-1}\|A\|^{-1}\mu_\omega(\partial_A\Omega)\,.\]
We start by showing that $\pp(\Omega)\lesssim (\frac{\Theta_\omega}{\vartheta_\omega})^3\vr^{2d} \mu_\omega(\partial_A\Omega)$. Since $1_\Omega$ is integrable and Lipschitz (we work on a uniformly discrete space), it suffices to show that 
    \begin{align}
        \fint_{B_r(x)} |1_\Omega-(1_\Omega)_{B_r(x)}| \,d\mu_\omega \lesssim  r \left(\frac{\Theta_\omega}{\vartheta_\omega}\right)^3 \vr^{2d} \fint_{B_{\lambda r}(x)} 1_{\partial_A \Omega}\, d\mu_\omega \,,\notag
        \intertext{or, equivalently,}
        \frac{\mu_\omega(\Omega\cap B_r(x))\mu_\omega(\Omega^c\cap B_r(x))}{\mu_\omega(B_r(x))^2}\lesssim  r{\left(\frac{\Theta_\omega}{\vartheta_\omega}\right)^3}\vr^{2d} \frac{\mu_\omega(\partial_A\Omega\cap B_{\lambda r}(x))}{\mu_\omega( B_{\lambda r}(x))}\,,\label{eqball}
    \end{align}
  for some suitable $\lambda\ge 1$ and any ball $B_r(x)$ with $r \ge \sigma_1(A)= \vr^{-1}$ (otherwise the left-hand side is trivially 0).

    In order to apply Theorem~\ref{theoiso}, we first work with cubes in $\Z^d$ whose images under $A$ correspond to parallelepipeds rather than Euclidean balls. Consider $\nu$ the uniform probability distribution over $Q=u+\{1,\ldots,k\}^d$ where $u\in\Z^d$, $k\in\N$. By \eqref{eqI} and Theorem~\ref{theoiso},
    \begin{align}
    \label{eqper}
        \nu(A^{-1}\Omega)\nu(A^{-1}\Omega^c)\le \sqrt{\pi} \ee\left[\Big(\sum_{j=1}^d\vv_j(1_{A^{-1}\Omega})\Big)^{1/2}\right]
        \le \sqrt{\pi} \sum_{j=1}^d\ee\left[\sqrt{\vv_j(1_{A^{-1}\Omega})}\right]\,.
    \end{align}
Assume $j=1$ for notational convenience and let us estimate $\vv_1(1_{A^{-1}\Omega})(x^{(1)})$, where we take the variance over the first variable $x_1$ and leave $x^{(1)}=(x_2,\ldots,x_k)$ fixed. Denoting the $x^{(1)}$-section of a set $M\subseteq\Z^d$ by $M_{x^{(1)}}=\{x_1\in\Z: \ (x_1,x^{(1)})\in M\}$ and the expected value with respect to the first variable by $E_1$, we get
\begin{multline*}
    \vv_1(1_{A^{-1}\Omega})=
\ee_1(1_{A^{-1}\Omega})-\ee_1^2(1_{A^{-1}\Omega})
\\ \le \begin{cases}
    0 & \text{if } (A^{-1}\Omega)_{x^{(1)}}\cap (u_1+\{1,\ldots,k\})=\emptyset  \text{ or }  u_1+\{1,\ldots,k\}
    \\ \tfrac{1}{4}  & \text{otherwise}
\end{cases}\,,
\end{multline*}
where we used that $t-t^2\le\tfrac{1}{4}$ for $t\in[0,1]$. In particular,
\[\vv_1(1_\Omega)\le \frac{(\#(\partial_{\id} A^{-1}\Omega\cap Q)_{x^ {(1)}})^2}{4}  \,.\]
Taking the expectation of the square root we arrive at
\begin{align*}
    \ee\Big[\sqrt{\vv_1(1_{A^{-1}\Omega})}\Big]&\le \tfrac{1}{2} \ee\big[\#(\partial_{\id} A^{-1}\Omega\cap Q)_{x^{(1)}}\big]=\tfrac{1}{2} \ee\left[\sum_{x_1=u_1+1}^{u_1+k}1_{\partial_{\id} A^{-1}\Omega}(x_1,x^{(1)})\right]
    \\&=\frac{k}{2} \nu(\partial_{\id}A^{-1}\Omega)=\frac{k}{2} \,\frac{\mu_\Lambda(\partial_{A}\Omega\cap AQ)}{\mu_\Lambda(AQ)} \,.
\end{align*}
Naturally, the same works for any $j$, so plugging this into \eqref{eqper} we obtain
\begin{align}
\label{eq_poin_disc}
    \frac{\mu_\Lambda(\Omega\cap AQ)\mu_\Lambda(\Omega^c\cap AQ)}{\mu_\Lambda(AQ)^2}=\nu(A^{-1}\Omega)\nu(A^{-1}\Omega^c)\le\frac{\sqrt{\pi} d k}{2}  \, \frac{\mu_\Lambda(\partial_A\Omega\cap AQ)}{\mu_\Lambda(AQ)}\,.
\end{align}
To obtain \eqref{eqball} from \eqref{eq_poin_disc} we need to go back from parallelepipeds to balls. Let us first check that  $(\Lambda,\mu_\Lambda,\dd)$ is doubling with a doubling constant $\le 5^{d}$. To see this, fix $x\in\Lambda$ and $r>0$. Cover $B_{2r}(x)\cap\Lambda$ with balls $B_r(x_j)\cap\Lambda$, where $1\le j\le J$, $x_j\in B_{2r}(x)\cap\Lambda$ and $\dd(x_j,x_i)\ge r$ for $j\neq i$ (here $B_s(x)\subseteq \R^{d}$ denotes the Euclidean ball in $\R^{d}$, not the discrete ball in $\Lambda$).
Notice that the balls $B_{r/2}(x_j)$ are disjoint and included in $B_{5r/2}(x)$. So, 
\[J|B_{1}(0)|\big(\tfrac{r}{2}\big)^{d}\le |B_{1}(0)|\big(\tfrac{5r}{2}\big)^{d}\,.\]
In particular,
\begin{align}
    \label{eq:countingdoubling}
    \#(B_{2r}(x)\cap\Lambda)\le J\,\#(B_{r}(x)\cap\Lambda)\le 5^{d}\,\#(B_{r}(x)\cap\Lambda)\,.
\end{align}
From this, we see that $\mu_\Lambda$ satisfies \ref{doubling} with constant $\le 5^d$.
Inequality \eqref{eqball} for $\mu_\Lambda$ now follows from \eqref{eq_poin_disc} by choosing $\lambda=C_d \vr$ where $C_d\ge1$ is such that one can inscribe $B_r(x)\subseteq AQ\subseteq B_{\lambda r}(x)$ for some suitable cube $Q$, and noticing that $\mu_\Lambda(B_{\lambda r}(x))\lesssim \lambda^d \mu_\Lambda(B_r(x))$.   Using \eqref{eq:comp-meas}, we get \eqref{eqball} for $\mu_\omega$.

It remains to show that {$\mu_{\omega}(\partial_A\Omega)\lesssim (\frac{\Theta_\omega}{\vartheta_\omega})^5\pp(\Omega)$.} 
Notice that by  \eqref{eq:countingdoubling} and \eqref{eq:comp-meas}, the measure $\mu_\omega$ is doubling with a doubling constant $\le 5^{d}\Mf/\mf$. So, applying Lemma~\ref{lem:p} to $\varphi(x,y)=\delta_1(|A^{-1}(x-y)|)$ we get 
\begin{align*}
    \mu_\omega(\partial_A\Omega) &\leq \sum_{x\in\partial_A\Omega} \omega_x \sum_{y\in\Omega^c}\varphi(x,y) \leq \vartheta_\omega^{-1} \int_\Omega \int_{\Omega^c} \varphi(x,y) \,d\mu_\omega(x) \,d\mu_\omega(y) \\&
    \lesssim \frac{\Theta_\omega^4}{\vartheta_\omega^5} \pp(\Omega) \sum_{n\geq0} \sup_{y\in\Lambda} \int_{A_{n,y}}(1+|x-y|)\delta_1(|A^{-1}(x-y)|)\,d\mu_\omega(x) \\
    &\leq \frac{\Theta_\omega^5}{\vartheta_\omega^5} \pp(\Omega) \sum_{x\in\Lambda}(1+|x|)\delta_1(|A^{-1}x|)
    \\ &\lesssim \frac{\Theta_\omega^5}{\vartheta_\omega^5}  \pp(\Omega) \#\{u\in\Z^d:\ |u|=1\} \max_{u\in\Z^d: \ |u|=1}(1+|Au|)\lesssim  \frac{\Theta_\omega^5}{\vartheta_\omega^5} \pp(\Omega)\,. \qedhere
\end{align*}
\end{proof}

\subsection{Stability under discretization}
\label{sec_stab}

For a  compact set $\Omega\subseteq \R^d$ and a (full-rank) lattice $\Lambda$ we define the discretized set $\Omega_\Lambda=  \Omega\cap \Lambda$ regarded as a subset of $(\Lambda,\mu_\Lambda,\dd)$, where as before $\mu_\Lambda$ denotes the normalized counting measure and $\dd$ is the Euclidean metric. For $\Omega$ with regular boundary we can estimate the perimeter $\pp(\Omega_\Lambda)$ and the constant $\ii(\Omega_\Lambda)$ from \ref{stripe} in terms of $\HH(\partial\Omega)$ and the regularity parameters. Importantly, the following bounds are stable under isotropic contractions of the lattice, cf. Remark~\ref{rem_isof}.

\begin{lemma}\label{th:gabor_distance}
    Let $\Omega \subseteq \R^{d}$ be a compact  set with regular boundary at scale $\eta>0$ and constant $\kappa$. Consider $\Lambda=A\mathbb{Z}^d$ with $A\in \gl_d(\R)$ endowed with the Euclidean distance 
    $\mathrm{d}$ and a weighted measure $\mu_\omega$, cf. \eqref{eq:weighted_measure}. Then, \ref{stripe} is satisfied for $\Omega_\Lambda=\Omega\cap\Lambda$ as a subset of $\Lambda$ with $\gamma=d$ and
    \begin{align}
   \label{eq_mlml1}    
    \ii(\Omega_\Lambda)&\lesssim \max\{1,\|A\|^{d}\} \frac{\Theta_\omega}{|\Lambda|}\frac{\HH(\partial \Omega)}{\kappa\, \min\{1,\eta^{d-1}\}}\,, \qquad \text{and}
    \\  \label{eq_mlml3}\pp(\Omega_\Lambda)&\lesssim \isof_\Lambda \cdot \frac{\Theta_\omega^4}{\vartheta_\omega^3|\Lambda|} \cdot 
    \frac{\HH(\partial \Omega)}{\kappa\, \min\{1,\eta^{d-1}\}}\,.
    \end{align}
    In addition, for $\mu_\omega=\mu_\Lambda$,
    \begin{equation}\label{eq:comp-Omega-Omega_Lambda}
    \big|\mu_\Lambda(\Omega_\Lambda)-|\Omega|\big|\lesssim   \|A\|\cdot\max\{1,\|A\|^{d-1}\} \cdot\frac{\HH(\partial \Omega)}{\kappa\, \min\{1,\eta^{d-1}\}}\,.
    \end{equation}
     The implied constants depend only on the dimension $d$. 
\end{lemma}

\begin{proof}
\noindent {\bf Step 1}. We start by showing \eqref{eq_mlml1}.
From Proposition~\ref{reduct} we know that 
\begin{equation}\label{eq:measure-strip}
\big|\{x\in \R^{d}:\ \mathrm{d}(x,\partial\Omega)\le 2^n\}\big|\lesssim \frac{\HH(\partial\Omega)}{\kappa\, \min\{1,\eta^{d-1}\}}\, 2^{ dn}, \qquad n\in\N_0\,, 
\end{equation}
since $\R^{d}$ is convex and the Lebesgue measure is doubling with constant $C_{\R^{d}}=2^{d}$.

Let us   define $Q=[-1/2,1/2]^{d}$ and show that  
  \begin{align*}
          \{\lambda\in  \Omega_\Lambda^c: \   \mathrm{d}(\lambda,\Omega_\Lambda)\leq 2^n \}+AQ &
        \subset \{x\in \R^{d} : \ \mathrm{d}(x,\partial\Omega)\le 2^{k+1} \}\,,
       \end{align*}
       where $k=\max\{n,\lceil\log_2(\sqrt{d} \|A\|)\rceil\}$. 
If $\mathrm{d}(\lambda,\lambda')\leq 2^n$ with $\lambda\in\Omega_\Lambda^c$ and $\lambda'\in\Omega_\Lambda$, there is $x\in \partial\Omega$ such that $\mathrm{d}(\lambda,x)\leq 2^n$. So if $y\in AQ$, then 
\[\mathrm{d}(\lambda+y,x)\leq \|y\|+2^n\le \sqrt{d}\|A\| + 2^n\le 2^{k+1}\,.\]
        So, by a comparison of  volumes it thus follows from \eqref{eq:measure-strip} that
   \begin{align*}
          |\Lambda|\#\{\lambda\in  \Omega_\Lambda^c: \  \mathrm{d}(\lambda,\Omega_\Lambda)\le 2^n \}  &
        \leq  \big|\{x\in \R^{d} : \  \mathrm{d}(x,\partial\Omega)\leq 2^{k+1} \}\big|
        \\
        &\lesssim  \frac{\HH(\partial\Omega)}{\kappa\, \min\{1,\eta^{d-1}\}}\,2^{ d(k+1)}\lesssim  \frac{\HH(\partial\Omega)}{\kappa\, \min\{1,\eta^{d-1}\}}\, \max\{1,\|A\|^{d}\}\,2^{ dn}\,.
       \end{align*}
Together with \eqref{eq:comp-meas}, this provides a bound related to \ref{stripe} and the set $E=\Omega_\Lambda$, with the desired constants. An analogous argument applies to $E=\Omega_\Lambda^c$ (by estimating the number of elements in 
$   \{\lambda\in  \Omega_\Lambda: \  \mathrm{d}(\lambda,\Omega_\Lambda^c)\leq 2^n \}$).

\smallskip

\noindent {\bf Step 2}. Regarding \eqref{eq_mlml3},
by Proposition~\ref{lemdisper} and the comparison \eqref{eq:comp-meas},
\[\text{Per}(\Omega_\Lambda)\lesssim  \vr^{2d} \|A\|^{-1} \frac{\Theta_\omega^4}{\vartheta_\omega^3}\#(\partial_A\Omega_\Lambda)\,.\]
Since $\partial_A\Omega_\Lambda+AQ\subset \partial\Omega+B_{\sqrt{d} \|A\|}(0)$, we deduce by \eqref{eqred2} that  
\begin{align*}
\pp(\Omega_\Lambda)&\lesssim \frac{\Theta_\omega^4}{\vartheta_\omega^3|\Lambda|}  \vr^{2d}\|A\|^{-1} |\Lambda| \#(\partial_A\Omega_\Lambda)\leq \frac{\Theta_\omega^4}{\vartheta_\omega^3|\Lambda|}\vr^{2d}\|A\|^{-1}|\partial\Omega+B_{\sqrt{d} \|A\|}(0)|
\\ &\lesssim \frac{\Theta_\omega^4}{\vartheta_\omega^3|\Lambda|} \vr^{2d} \max\{1,\|A\|^{d-1}\} \frac{\HH(\partial\Omega) }{\kappa\min\{1,\eta^{d-1}\}}\,. 
\end{align*}
 \smallskip
\noindent {\bf Step 3}. Equation \eqref{eq_mlml3} follows from  \eqref{eq_mlml1} and \eqref{eq_mlml3}. To prove \eqref{eq:comp-Omega-Omega_Lambda}, we let $\Omega_\Lambda^+:=\{\lambda\in\Lambda: \ \Omega\cap (\lambda+AQ)\neq \emptyset\}$ and $\Omega_\Lambda^-:=\Omega_\Lambda\smallsetminus \partial_A\Omega_\Lambda.$
Then by  Proposition~\ref{reduct}
\begin{align*}
|\Lambda|\#\Omega_\Lambda&= |\Lambda|\#\Omega_\Lambda^-+|\Lambda|\# (\partial_A\Omega_\Lambda)\leq |\Omega|+\big|\partial\Omega+B_{2\sqrt{d}\|A\|}(0)\big|\\&\leq |\Omega|+ C\|A\|\max\{1,\|A\|^{d-1}\} \frac{\HH(\partial\Omega) }{\kappa\min\{1,\eta^{d-1}\}}\,.
\end{align*}
Similarly,  
\begin{align*}
|\Lambda|\#\Omega_\Lambda&= |\Lambda|\#\Omega_\Lambda^+-|\Lambda|\#  \big(\Omega_\Lambda^+\smallsetminus\Omega_\Lambda\big)\geq |\Omega|-\big|\partial\Omega+B_{2\sqrt{d}\|A\|}(0)\big|
\\
&\ge  |\Omega|- C\|A\|\max\{1,\|A\|^{d-1}\} \frac{\HH(\partial\Omega) }{\kappa\min\{1,\eta^{d-1}\}}\,. \qedhere
\end{align*}
\end{proof}

\section{Spectral deviation for concentration operators}\label{sec:deviation}
\subsection{Schatten-norm estimates for Hankel operators}\label{sec:hankel}
The next proposition provides a bound for the Schatten $p$-norms of the Hankel operator $H=H_\Omega=(1-P) 1_\Omega P$. Crucially, we quantify the dependence on $p$ for $p \to 0^+$. With the set of tools developed in Section~\ref{sec_bo} at hand, we shall parallel and aptly adapt the argument of \cite[Proposition 3.1]{MR}. Recall the dyadic decay measure $N(s)$ from \eqref{eqn}.

\begin{proposition}\label{th:H_norm}
    Let $\Omega\subseteq X$ be a compact set, and let $H$ be the corresponding Hankel operator. Assume Conditions~\ref{norm2} and \ref{stripe} hold and let $p \in(0,2]$ and $\alpha\in (0,1]$.
    Then
    \begin{align} \label{eq:H0}
	\lVert H \rVert_{\S_p}^p 
        \lesssim    \alpha^{-1+\tfrac{p}{2}}\ii(\Omega) \,N\big((\gamma+\alpha)\tfrac{2-p}{p}+\gamma\big)^{\tfrac{p}{2}} \,.
    \end{align}
If in addition Condition~\ref{doubling} is met, 
then one also has
    \begin{align} \label{eq:H}
	\lVert H \rVert_{\S_p}^p 
        \lesssim    {C_X^{2p}}  \left(\frac{\ii(\Omega)}{\alpha}\right)^{1-\tfrac{p}{2}}\pp(\Omega)^{\tfrac{p}{2}} \, N\big((\gamma+\alpha)\tfrac{2-p }{p}+1\big)^{\tfrac{p}{2}} \,.
    \end{align}
\end{proposition}

\begin{proof} 
We assume \ref{norm2} and first check that 
	\begin{align*}
	    \|H\|_{\S_p}^p &\le \int_X \|HK_x\|_{\S_2} ^{p}  \, d\mu(x)\,.
	\end{align*}
    Indeed, let $\{\psi_k\}_{k\in I}\subset L^2(X)$ and $\{g_n\}_{n\in J}\subset \HH$ be orthonormal bases. Then $\{\psi_k g_n\}_{(k,n)\in I\times J}$ forms an orthonormal basis for $L^2(X,\HH)$.
Since $K$ is the reproducing kernel of  $\mathbb{H}\subset L^2(X,\HH)$ it follows that for $f,g\in\HH$ we have that $\langle f,K(x,\, \cdot\,)g\rangle \in L^2(X)$ which implies that for almost every $x,y\in X$
\begin{align*}
\langle f,K(x,y)g\rangle
&=\sum_{k\in I}\int_X\big\langle f,K(x,x') {\psi_k(x')}\overline{\psi_k(y)}g\big\rangle  \,d\mu(x')\,.
\end{align*}
For $A=(H^\ast H)^{p/2}$, using \cite[Lemma 2.1]{MR} we obtain
    \begin{align*}
   \|&H\|_{\S_p}^p  =\text{trace}(A)=\sum_{(k,n)\in I\times J} \langle A( \psi_k g_n),\psi_k g_n\rangle_{L^2(X, \HH)}
    \\ 
&=\sum_{k,n}\int_X \langle A( \psi_k g_n)(x),\psi_k(x)g_n\rangle \,d\mu(x)
\\
&=\sum_{k,n}\int_X\int_X \big\langle A (K_y\psi_k(y)g_n)(x),\psi_k(x)g_n\big\rangle\, d\mu(y)d\mu(x)
\\
&=\sum_{k,n}\int_X\int_X \big\langle A (K_yg_n)(x),\overline{\psi_k(y)}\psi_k(x)g_n\big\rangle \,d\mu(y)d\mu(x)
\\
&=\sum_{k,n}\int_X\int_X \int_X\big\langle A (K_y g_n)(x),K(x,x')\overline{\psi_k(y)}\psi_k(x')g_n\big\rangle \,d\mu(x') d\mu(y)d\mu(x)
\\
&=\sum_{n\in J}\int_X\int_X  \big\langle A (K_yg_n)(x),K(x,y) g_n\big\rangle\, d\mu(y)d\mu(x)
\\
&=\sum_{n\in J}\int_X\|K_yg_n\|_{L^2(X,\HH)}^2\int_X  \left\langle A \left(\frac{K(\, \cdot\, , y)g_n }{\|K_yg_n\|_{L^2(X,\HH)}}\right)(x),\frac{K(x,y)g_n }{\|K_yg_n\|_{L^2(X,\HH)}}\right\rangle \,d\mu(x)d\mu(y)
\\
&\leq\sum_{n\in J}  \int_X\|K_yg_n\|_{L^2(X,\HH)}^2  \\
&\hspace{2cm}\cdot\left(\int_X \left \langle H^\ast H\left(\frac{K(\, \cdot\, , y)g_n }{\|K_yg_n\|_{L^2(X,\HH)}}\right)(x),\frac{K(x,y)g_n }{\|K_yg_n\|_{L^2(X,\HH)}}\right\rangle \,d\mu(x)\right)^{\frac{p}{2}}\,d\mu(y)
\\
&=\sum_{n\in J}  \int_X \|K_yg_n\|_{L^2(X,\HH)}^{2-p} \big \|H( K_yg_n)\big\|^{p}_{L^2(X, \HH)}\,d\mu(y) 
\\
&\leq \int_X\left(\sum_{n\in J}  \|K_yg_n \|^{2}_{L^2(X, \HH)}\right)^{1-\frac{p}{2}} \left(\sum_{n\in J}\big \|H \big(K_yg_n\big)   \big\|^{2}_{L^2(X, \HH)}\right)^{\frac{p}{2}}\,d\mu(y) 
\\
&= \int_X   \|K_y  \|_{L^2(X, S_2)}^{2-p}  \|H K_y   \|_{L^2(X,\S_2)}^{p}\,d\mu(y) 
\\
&\leq {}\int_X    \|H K_y    \|_{L^2(X, \S_2)}^{p}\,d\mu(y)\,,
\end{align*}
where we used the assumption \ref{norm2} in the final step.
    
    Notice that $H^* H=T_\Omega T_{\Omega^c}\preceq T_E$, where $E$ denotes either $\Omega$ or $\Omega^c$. So,
        \begin{align*}
    \|HK_y\|_{L^2(X,\S_2)}^2 &= \sum_{n\in J}\int_X\langle H^* H( K_yg_n)(x'), K_y(x')g_n \rangle \,d\mu(x')
    \\
    &\leq \sum_{n\in J}\int_X\langle T_E (K_yg_n)(x'), K_y(x')g_n \rangle \, d\mu(x')
      \\
    &=\sum_{n\in J}\int_X\int_E\langle K(y,x')K(x',x) K(x ,y)g_n,g_n \rangle \,d\mu(x) d\mu(x')
          \\
&=\int_E \|K(x,y)\|_{\S_2}^2\,d\mu(x)\,.
    \end{align*}
    In particular,
        \begin{align} \label{eq2r}
	    \|H\|_{\S_p}^p\le \int_{\Omega} \Big(\int_{\Omega^c} \|K(x,y)\|_{\S_2}^2\, d\mu(y)\Big)^{\frac{p}{2}} \, d\mu(x)+\int_{\Omega^c} \Big(\int_{\Omega} \|K(x,y)\|_{\S_2}^2\, d\mu(y)\Big)^{\frac{p}{2}} \, d\mu(x)\,.
	\end{align}
                Recall the parameter $\alpha\in (0,1]$ and for $E=\Omega$ or $\Omega^c$ define 
        \[\psi(x)=(1+\mathrm{d}(x,E^c))^{\gamma+\alpha} \,, \qquad x\in X\,. \]
By H\"{o}lder's inequality,
	\begin{multline}\label{i3r}
            \int_{E}  \Big(\int_{E^c} \|K(x',x)\|_{\S_2}^2 \, d\mu(x')\Big)^{\frac{p}{2}}  \, d\mu(x) 
            \\ 
            \le \Big( \int_{E} \frac{1}{\psi(x)} \, d\mu(x)\Big)^{1-\frac{p}{2}} \Big(\int_{E}\int_{E^c} \psi(x)^{\frac{2-p}{p}} \|K(x',x)\|_{\S_2}^2 \, d\mu(x')\,d\mu(x)\Big)^{\frac{p}{2}} .
	\end{multline}
        For the first integral, using Condition~\ref{stripe} and the assumption $\alpha \le1$ we obtain
	\begin{align}\label{coar}
            \int_{E} \frac{1}{\psi(x)} \, d\mu(x) &
            = \int_{\{x\in E : \  \mathrm{d}(x,E^c)\le 1\}} \frac{1}{\psi(x)} \, d\mu(x) +\sum_{n\in\N} \int_{\{x\in E : \ 2^{n-1}<  \mathrm{d}(x,E^c)\le 2^{n}\}} \frac{1}{\psi(x)} \, d\mu(x) \\&
            \le \sum_{n\in\N_0} {2^{-(\gamma+\alpha)(n-1)}}\mu(\{x\in E: \ \mathrm{d}(x,E^c)\le 2^n\}) \notag\\&
            \le  \ii(\Omega) \sum_{n\in\N_0} 2^{-\alpha (n-1)}
            \lesssim \ii(\Omega) \frac{1}{2^\alpha-1}\lesssim \frac{\ii(\Omega)}{\alpha}\,. \notag
	\end{align}
        For the second integral,
        \begin{align*}
            \int_{E}\int_{E^c} \psi(x)^{\frac{2-p}{p}} \|K(x',x)\|_{\S_2}^2& \, d\mu(x')\,d\mu(x) \\&
            \le \int_{E}\int_{E^c} (1+\mathrm{d}(x,x'))^{(\gamma+\alpha)\frac{2-p }{p}}\|K(x',x)\|_{\S_2}^2 \, d\mu(x')\,d\mu(x) \\&
            = \int_{\Omega}\int_{\Omega^c} (1+\mathrm{d}(x,x'))^{(\gamma+\alpha)\frac{2-p }{p}}\|K(x',x)\|_{\S_2}^2 \, d\mu(x')\,d\mu(x) \,.
        \end{align*}
        Combining this with \eqref{eq2r}, \eqref{i3r} and \eqref{coar} we get
        \begin{align*}
            \|H\|_{\S_p}^p \lesssim \left(\frac{\ii(\Omega)}{ \alpha}\right)^{1-\tfrac{p}{2}}\left(\int_{\Omega}\int_{\Omega^c} (1+\mathrm{d}(x,x'))^{(\gamma+\alpha)\frac{2-p }{p}}\|K(x',x)\|_{\S_2}^2 \, d\mu(x')\,d\mu(x) \right)^{\frac{p}{2}}\,.
        \end{align*}
        The result now follows from applying Lemmas~\ref{lem:p2} and \ref{lem:p} to obtain \eqref{eq:H0} and \eqref{eq:H} respectively. For the latter we also use that $\|K(x',x)\|_{\S_2}=\|K(x,x')\|_{\S_2}$. 
\end{proof}

\subsection{Proof of the main result}\label{subsec:deviation}

\begin{proof}[Proof of Theorem~\ref{th:general_bound}]
    The spectral deviation of $T_\Omega$ can be controlled in terms of the Hankel operator $H$ by functional calculus with the function $(t-t^2)^{p/2}$, which gives
    \begin{align*}
        \big| \#\{\lambda \in \sigma(T_{\Omega}): \lambda > \delta\} -  {  }\mu(\Omega) \big|
    \le 2\tau^{\frac{p}{2}} \|H\|_{\S_p}^p,
    \end{align*}
    where $\tau = \max\{\tfrac{1}{\delta},\tfrac{1}{1-\delta}\}$;
    see, e.g., \cite[Lemma 4.1]{MR} for details. Combining this with Proposition~\ref{th:H_norm}, we conclude that for $\alpha\in (0,1]$ and $p\in (0,2]$,
    \begin{align}\label{eqplu}
        \big| \#\{\lambda \in \sigma(T_{\Omega}): \lambda > \delta\} -  {  }\mu(\Omega) \big|
    \lesssim     \alpha^{-1+\tfrac{p}{2}}\ii(\Omega)  \tau^{\frac{p}{2}} N\big((\gamma+\alpha)\tfrac{2-p}{p}+\gamma\big)^{\tfrac{p}{2}} \,.
    \end{align}
        Without loss of generality, we can assume that there is some $s\ge\gamma$ such that $N(s)<\infty$, since otherwise the theorem  holds trivially. 

    For $s\ge\gamma$ with $N(s)<\infty$, define
    \begin{align}\label{eqap}
        \alpha=1/\log^{\ast}\big((\tau {}N(s))^{1/s}\big) 
        \,,
        \qquad \text{and} \qquad
        p=2\frac{\gamma+\alpha}{s+\alpha}\,,
    \end{align}
    so that the function $N$ on the right-hand side of \eqref{eqplu} is evaluated at $s$.
  Since $2\le \tau N(s)<\infty$, it follows that $\alpha\in (0,1]$ and $p\in (0,2]$.
  Thus, $p$ and $\alpha$ are valid choices in the sense that  Proposition~\ref{th:H_norm} is applicable.
   Plugging the values from \eqref{eqap} into \eqref{eqplu} we get
    \begin{align*}
        \big| \#\{\lambda \in \sigma(T_{\Omega}): \lambda > \delta\} - \mu(\Omega) \big| 
        &\lesssim  \ii(\Omega)\,    { }  \big(\tau   { }N(s)\big)^{\tfrac{\gamma+\alpha}{s+\alpha} } \log^\ast\Big(\big(\tau {}N(s)\big)^{\tfrac{1}{s}}\Big)^{\tfrac{s-\gamma}{s+\alpha} }
       \\ &
        \lesssim  \ii(\Omega)\,  { } \big(\tau  {}N(s)\big)^{\tfrac{\gamma}{s}+\tfrac{\alpha}{s} } \log^\ast\Big(\big(\tau  {}N(s)\big)^{\tfrac{1}{s}}\Big)^{1-\tfrac{\gamma}{s} }
        \\ &
        \lesssim  \ii(\Omega)\,  { } \big(\tau  {}N(s)\big)^{\tfrac{\gamma}{s} } \log^\ast\Big(\big(\tau  {}N(s)\big)^{\tfrac{1}{s}}\Big)^{1-\tfrac{\gamma}{s} }\,.
    \end{align*}
    Taking the infimum over $s\ge\gamma$, we obtain \eqref{eq:last_common}. The proof of \eqref{eq:last_common2} is analogous once we bound 
    \[\ii(\Omega)^{1-\tfrac{p}{2}}\pp(\Omega)^{\tfrac{p}{2}}\le \max\{\ii(\Omega),\pp(\Omega)\}\,.\qedhere\]
\end{proof}

\begin{proof}[Proof of Corollary~\ref{lemma:exp}]
For $s\ge 1$,
    \begin{align*}
(1+ \mathrm{d}(x,y))^s e^{-\alpha  \mathrm{d}(x,y)^{1/\beta}} \le \sup_{r\geq 0} (1+r)^s e^{-\alpha r^{1/\beta}}.
\end{align*}
The last expression is at most $2^s$ if $r\le 1$, whereas for $r>1$ 
\[(1+r)^s e^{-\alpha r^{1/\beta}}\le 2^s r^s e^{-\alpha r^{1/\beta}}\,,\]
which attains its maximum at $r=(\beta s/\alpha)^{\beta}$. Hence, 
\begin{align*}
N(s)
&\lesssim \sup_{y\in X} \int (1+\mathrm{d}(x,y))^{s+1} e^{-\alpha \mathrm{d}(x,y)^{1/\beta}} e^{\alpha \mathrm{d}(x,y)^{1/\beta}} \|K(x,y)\|_{\S_2}^2\, dx \\
&\leq {D_\hh} 2^{s+1} \max\left\{1, \big(\tfrac{\beta}{\alpha}\big)^{\beta (s+1)}\right\} (s+1)^{\beta (s+1)} 
\\ & \leq  {D_\hh} 2^{(\beta+1)2s} \max\left\{1, \big(\tfrac{\beta}{\alpha}\big)^{2\beta s }\right\} s^{ {\beta (s+1)}} \,.
\end{align*}
Now choose $s=\log^\ast(\tau {D_\hh} )$ and assume without loss of generality that $\log(\tau)\ge \gamma$ so that $s\ge \gamma$. We get
\begin{align*}
(\tau {} N(s))^{\tfrac{1}{s}}&\le 4^{\beta+1} \max\left\{1, \big(\tfrac{\beta}{\alpha}\big)^{2\beta}\right\} \big(\tau {D_\hh} \log^\ast(\tau {D_\hh}  )^\beta\big)^{\log^{\ast} (\tau {D_\hh}  ) ^{-1}} \log^{\ast}(\tau {D_\hh}  )^{ {\beta }}
\\
&\lesssim \log^{\ast}(\tau {D_\hh}  )^{ {\beta }}\,. 
\end{align*}
The result then follows from Theorem~\ref{th:general_bound}, bounding the infimum in \eqref{eq:last_common} by the value corresponding to our  particular choice of $s$.
\end{proof}

\section{Toeplitz operators and frame multipliers}\label{sec_frames}

For a separable Hilbert space $\mathcal{H}$ and a full-rank lattice $\Lambda\subset\R^d$, consider a \emph{frame} $\{\varphi_\lambda\}_{\lambda\in\Lambda}\subset \mathcal{H}$, that is, there are constants   $a,b>0$ such that
\begin{align}
    \label{eqframe}
  a\, \|f\|^2\le \sum_{\lambda\in\Lambda}|\langle f,\varphi_\lambda\rangle|^2\le  b\,  \|f\|^2,\qquad f\in\mathcal{H} \,.
\end{align}
We will assume throughout this section  that $\varphi_\lambda\neq 0$ for every $\lambda\in\Lambda$.
Let $S_\varphi:\mathcal{H}\to\mathcal{H}$ be the \emph{frame operator} given by
\[S_\varphi f=\sum_{\lambda\in\Lambda}\langle f,\varphi_\lambda\rangle \varphi_\lambda, \qquad f\in\mathcal{H} \,.\]
From \eqref{eqframe} it follows that $S_\varphi$ is invertible and, defining the \emph{canonical dual frame} as $\ddual_\lambda=S^{-1}_\varphi\varphi_\lambda$, we get the inversion formula
\begin{align}
    \label{eq_invframe}
    f=\sum_{\lambda\in\Lambda}\langle f,\varphi_\lambda\rangle \ddual_\lambda,\qquad f\in\mathcal{H} \,.
\end{align}
Given a compact set $\Omega\subset\R^d$, consider its associated \emph{frame multiplier} $M_{\varphi,\Omega}:\mathcal{H}\to\mathcal{H}$ given by
\begin{align}\label{eq_fm}
M_{\varphi,\Omega} f = \sum_{\lambda \in \Lambda \cap \Omega} \langle f, \varphi_\lambda \rangle \ddual_\lambda,\qquad f\in\mathcal{H}\,.
\end{align}
Applying our main results Theorem~\ref{th:general_bound} and Corollary~\ref{lemma:exp} will allow us to derive the following spectral deviation bounds for frame multipliers.
A key quantity in our estimates is  $\omega_\lambda:=\langle \ddual_{\lambda}, \varphi_\lambda \rangle$, the diagonal of the cross-Gram matrix of $\varphi$ and $\ddual$. Using $\omega$, we define   $\mu_\varphi:=\mu_\omega$ via \eqref{eq:weighted_measure}, together with the extrema $\mf$ and $\Mf$. Moreover, we set $
   C_\varphi := {\Mf^8}/({\mf^7 \,|\Lambda|})
$
and note that 
\[b^{-1}\|\varphi_\lambda\|^2\le\mf\le \omega_\lambda\le\Mf\le a^{-1}\|\varphi_\lambda\|^2\,,\qquad \lambda\in\Lambda\,.\]

\begin{theorem}\label{th:frame_conc}
Let $\Lambda = A\Z^{d}$ be a full-rank lattice with isotropic fineness $\isof_\Lambda$, cf. \eqref{eq_isof},
and $\Omega\subset\R^{d}$ a compact set with regular boundary at scale $\eta$ and constant $\kappa$. For a frame  $\{\varphi_\lambda\}_{\lambda\in\Lambda}$  for $\mathcal{H}$, consider the frame multiplier $M_{\varphi,\Omega}$ defined in \eqref{eq_fm}. If we assume that 
$$\frac{|\langle  \ddual_{\lambda'}, \varphi_\lambda \rangle|}{\sqrt{\omega_{\lambda'}\omega_{\lambda}}}\le u(\lambda-\lambda'),$$ for some $u:\R^d\to \R_{\ge 0}$,
 then, with the notation
$\delta\in(0,1)$, $\tau \coloneqq \max\{\frac{1}{\delta},\frac{1}{1-\delta}\} :$
\begin{enumerate}[label=(\roman*)]
\item\label{tfi} For every $s\ge 1$,
\begin{align*}
        \big|  \#\{\lambda \in \sigma(M_{\varphi,\Omega}):\  \lambda > \delta\} - {   } \mu_\varphi(\Omega) \big| \notag
       &\lesssim 
       \begin{multlined}[t]
       \isof_\Lambda \cdot C_\varphi
       \cdot\frac{\HH(\partial \Omega)}{\kappa\, \min\{1,\eta^{d-1}\}} \cdot      
       \\ \cdot  \big(\tau  N(s)\big)^{\frac{d}{s+d-1}}\Big(\log^\ast\Big(\big(\tau  N(s)\big)^{\tfrac{1}{s+d-1}}\Big)\Big)^{\tfrac{s-1}{s+d-1} }
       \,,
     \end{multlined}
\end{align*}
    where 
    \begin{align}
        N(s)=    
  \Mf
    \sum_{\lambda\in \Lambda} (1+ |\lambda|)^{s} u(\lambda)^2 \, .
    \end{align}
\item\label{tfii} For $\alpha,\beta>0$,
\begin{align*}
       \big|  \#\{\lambda \in \sigma(M_{\varphi,\Omega}):\  \lambda > \delta\} - {   } \mu_\varphi(\Omega) \big|
     \notag
       \lesssim 
           \isof_\Lambda &\cdot  \frac{\Mf}{|\Lambda|} \cdot \frac{\HH(\partial \Omega)}{\kappa\, \min\{1,\eta^{d-1}\}}  \cdot
           \\
          &\cdot  
          \big(\log^\ast(\tau D )\big)^{ d{\beta}} \cdot\log^\ast\big(\log^\ast(\tau D) \big)
            \,,
    \end{align*}
    where the implied constant depends on $\alpha,\beta,d$ and $D$ is given by
    \[D= \Mf \sum_{\lambda \in \Lambda} e^{\alpha |\lambda|^{1/\beta}} u(\lambda)^2
       \,.\]   
\end{enumerate}
\end{theorem}

\begin{proof}
Without loss of generality we assume $\mf>0$ as the claim holds trivially otherwise.

\noindent {\bf Step 1}  \emph{(Identification of a RKHS)}. 
First note that $\omega_\lambda=\langle S_\varphi^{-1}\varphi_\lambda,\varphi_\lambda\rangle =\|S_\varphi^{-1/2}\varphi_\lambda\|>0.$ Consider the analysis operators $C_{\varphi},C_{\ddual} \colon
\mathcal{H} \to L^2(\Lambda,\mu_\varphi)$ by
\[C_\varphi f(\lambda)=\frac{\langle f,\varphi_\lambda\rangle}{\sqrt{\omega_\lambda}}, \qquad \text{and}\qquad C_{\ddual} f(\lambda)=\frac{\langle f,\ddual_\lambda\rangle}{\sqrt{\omega_\lambda}} \,.\]
The range of $C_\varphi$ is the closed subspace
\begin{equation}\label{eq:rangeframe}
    \hh_\varphi \coloneqq \{C_\varphi f \colon f \in \mathcal{H}
    \} \subseteq L^2(\Lambda,\mu_\varphi)\,.
\end{equation}
For $f\in \mathcal{H}$, the inversion formula  \eqref{eq_invframe}
gives
\begin{align*}
    C_\varphi f (\lambda) 
    = \sum_{\lambda' \in \Lambda} \langle f, \varphi_{\lambda'} \rangle \frac{\langle  \ddual_{\lambda'}, \varphi_\lambda \rangle}{\sqrt{\omega_\lambda}}
    = \int_{\Lambda} C_\varphi f(\lambda')  \frac{\langle \ddual_{\lambda'}, \varphi_\lambda \rangle}{ \sqrt{\omega_\lambda\omega_{\lambda'}}}\, d\mu_\varphi (\lambda')
    \,,\qquad \lambda\in\Lambda\,.
\end{align*}
 So, $\hh_\varphi$ is a RKHS with kernel
\begin{equation*}
    K_\varphi(\lambda,\lambda') =  \frac{\langle \ddual_{\lambda'}, \varphi_\lambda \rangle}{\sqrt{\omega_\lambda\omega_{\lambda'}} }\,,
\end{equation*}
and orthogonal projection $P_\varphi:L^2(\Lambda,\mu_\varphi)\to \hh_\varphi$ given by $P_\varphi= C_\varphi C_{\ddual}^*$.

\noindent {\bf Step 3} \emph{(Associated concentration operator)}.
Let us consider the concentration operator $T_{\varphi,\Omega}:L^2(\Lambda,\mu_\varphi)\to L^2(\Lambda,\mu_\varphi)$ defined by
\begin{align*}
T_{\varphi,\Omega} v = P_\varphi \big( 1_{\Lambda\cap \Omega} \cdot P_\varphi v \big) \,, \qquad v \in L^2(\Lambda,\mu) \,.
\end{align*}
We can decompose $L^2(\Lambda,\mu_\varphi) = \hh_\varphi \bigoplus \hh_\varphi^\perp$. By \eqref{eq_fm} we have
\begin{align*}
    T_{\varphi,\Omega} =  \begin{bmatrix}
                    C_\varphi  M_{\varphi,\Omega}  C_{\ddual}^*  & 0\\
                    0 & 0
                \end{bmatrix}\,.
\end{align*}
From here it is easy to check that $M_{\varphi,\Omega}$ and $T_{\varphi,\Omega}$ have the same non-zero eigenvalues. 

\noindent {\bf Step 4} \emph{(Concentration estimate)}.
Let us check that Theorem~\ref{th:general_bound} and Corollary~\ref{lemma:exp} apply to $T_{\varphi,\Omega}$ and consequently to $M_{\varphi,\Omega}$.

It is clear that $\Lambda$ is a locally compact metric space with the Euclidean distance $\dd$ inherited from $\R^{d}$ and $\mu_\varphi$ is finite on compact sets. {As shown in \eqref{eq:countingdoubling}},
$\mu_\varphi$ satisfies \ref{doubling} with constant $\le 5^d \Mf/\mf$.

     On the other hand, Condition~\ref{stripe} follows from Lemma~\ref{th:gabor_distance}. 
     Regarding \ref{norm2}, 
\begin{equation*}
    \lVert K_{\lambda} \rVert_{L^2(\Lambda,\mu_\varphi)}^2 
    = \omega_\lambda^{-1}\sum_{\lambda' \in\Lambda}  |\langle \ddual_\lambda, \varphi_{\lambda'} \rangle |^2 =\omega_\lambda^{-1} \langle S \ddual_\lambda,\ddual_\lambda\rangle =1\,,\qquad \lambda\in\Lambda\,.
\end{equation*}
Recall from \eqref{eqn} that we denote $A_{n,\lambda}=B_{2^n}(\lambda)\smallsetminus B_{2^{n-1}}(\lambda)$ for $n\ge 1$, $A_{0,\lambda}=B_{1}(\lambda)$ and
\begin{align*}
    N(s) &=    \sum_{n\ge 0}\sup_{\lambda \in \Lambda}  \sum_{\lambda'\in A_{n,\lambda}} (1+ |\lambda-\lambda'|)^{s}  \frac{ |\langle \ddual_{\lambda'}, \varphi_\lambda \rangle|^2}{\omega_\lambda \omega_{\lambda'}}\omega_{\lambda'}
    \\ &\le \Mf \sum_{n\ge 0}\sup_{\lambda \in \Lambda}  \sum_{\lambda'\in A_{n,\lambda}} (1+ |\lambda-\lambda'|)^{s} u(\lambda-\lambda')^2
    \\ &\le \Mf   \sum_{\lambda\in \Lambda} (1+ |\lambda|)^{s} u(\lambda)^2
    \,.
\end{align*}
By Theorem~\ref{th:general_bound}, we see that for any $s\ge 1$,
 \begin{multline*}
        \big|  \#\{\lambda \in \sigma(M_{\varphi,\Omega}):\  \lambda > \delta\} - {   } \mu_\varphi(\Omega) \big| \notag
       \\ \lesssim \isof_\Lambda  \cdot \frac{\HH(\partial \Omega)}{\kappa\, \min\{1,\eta^{d-1}\}} \cdot \frac{\Mf^8}{\mf^7 \,|\Lambda|} \cdot \big(\tau  N(s)\big)^{\frac{d}{s+d-1}}\Big(\log^\ast\Big(\big(\tau  N(s)\big)^{\tfrac{1}{s+d-1}}\Big)\Big)^{\tfrac{s-1}{s+d-1} }\,.
    \end{multline*} 
Similarly, by Corollary~\ref{lemma:exp} and \eqref{eq_mlml1},
    \begin{multline*}
        \big|  \#\{\lambda \in \sigma(M_{\varphi,\Omega}):\  \lambda > \delta\} - {   } \mu_\varphi(\Omega) \big| \notag \\
       \lesssim \isof_\Lambda  \cdot \frac{\HH(\partial \Omega)}{\kappa\, \min\{1,\eta^{2d-1}\}} \cdot \frac{\Mf}{|\Lambda|} \cdot \big(\log^\ast(\tau D)\big)^{ 2d{\beta}} \cdot\log^\ast\big(\log^\ast(\tau D)\big)\,,
    \end{multline*}
    where, similar to before,
    \begin{align*}
        D&=\sup_{\lambda \in \Lambda} \sum_{\lambda' \in \Lambda} e^{\alpha |\lambda-\lambda'|^{1/\beta}} \frac{|\langle \ddual_{\lambda'},\varphi_\lambda\rangle|^2}{\omega_\lambda\omega_{\lambda'}}\omega_{\lambda'}
   \le  \Mf \sum_{\lambda \in \Lambda} e^{\alpha |\lambda|^{1/\beta}} u(\lambda)^2 \,. \qedhere
    \end{align*}    
\end{proof}

\begin{remark}\label{rem_tight}
If $\{\varphi_\lambda\}_{\lambda\in\Lambda}$ is a tight frame  (that is, $a=b$ in the frame inequality \eqref{eqframe}) consisting of unit-norm elements, then Theorem~\ref{th:frame_conc} simplifies and yields a bound concerning the spectral deviation from a multiple of $|\Omega|$. Indeed, the dual frame is simply $\ddual_\lambda=a^{-1}\varphi_\lambda$, which gives $\omega_\lambda\equiv a^{-1}$, and,
       \begin{align*}
        \big|  \#\{\lambda \in \sigma &(M_{\varphi,\Omega}):\  \lambda > \delta\} - (|\Lambda|a)^{-1}|\Omega| \big|
         \\ &\leq  \big|  \#\{\lambda \in \sigma(M_{\varphi,\Omega}):\  \lambda > \delta\} - {   } \mu_\varphi(\Omega) \big|+\big|\mu_\varphi(\Omega)-(|\Lambda|a)^{-1}|\Omega|\big|
         \\ &=  \big|  \#\{\lambda \in \sigma(M_{\varphi,\Omega}):\  \lambda > \delta\} - {   } \mu_\varphi(\Omega) \big|+(|\Lambda|a)^{-1}\big|\mu_\Lambda(\Omega)-|\Omega|\big|\,,\qedhere
         \end{align*}
which, combined with \eqref{eq:comp-Omega-Omega_Lambda} and Theorem \ref{th:frame_conc}\ref{tfi}, gives
\begin{align*}
        \big|  \#\{\lambda \in \sigma(M_{\varphi,\Omega}):\  \lambda > \delta\} - (a|\Lambda|)^{-1}  |\Omega| \big| \notag
       &\lesssim 
       \begin{multlined}[t]
       \isof_\Lambda 
       \cdot\frac{\HH(\partial \Omega)}{\kappa\, \min\{1,\eta^{d-1}\}} \cdot  (a|\Lambda|)^{-1} 
       \\ \cdot  \big(\tau  N(s)\big)^{\frac{d}{s+d-1}}\cdot\Big(\log^\ast\Big(\big(\tau  N(s)\big)^{\tfrac{1}{s+d-1}}\Big)\Big)^{\tfrac{s-1}{s+d-1} }
    \,,
     \end{multlined}
\end{align*}
    with $N(s)$ as before. The same simplification applies to Theorem \ref{th:frame_conc}\ref{tfii}.
Note that for typical examples (in particular when the vectors $\varphi_\lambda$ are obtained by sampling  a continuous tight frame indexed by $\R^d$) one expects the frame bound $a$ to scale like $|\Lambda|^{-1}$.   
\end{remark}

\section{Examples and applications in time-frequency analysis}\label{sec:application}

\subsection{Time-frequency concentration operators}\label{sec-stft}
Recall the time-frequency filter $\A^g$ from \eqref{eq_tf_dau}.
Heuristically, $\A^g$ is approximately a projection onto the space of functions whose short-time Fourier transforms are mainly localized on $\Omega$. Theorem~\ref{th:general_bound} and Corollary~\ref{lemma:exp}, which essentially recover the main results of \cite{MR}, help validate such 
intuition.

Precisely, we consider the $\hh_g=V_g(L^2(\mathbb{R}^d))\subseteq L^2(\mathbb{R}^{2d})$, which is a RKHS with kernel
\begin{align*}
K_g(z,w)=
V_gg (z-w) e^{2\pi i (\xi'-\xi) x'}\,, \qquad z=(x,\xi)\,, w=(x',\xi') \in \mathbb{R}^d \times \mathbb{R}^d\,.
\end{align*}
The STFT is an isometric isomorphism
$V_g: L^2(\mathbb{R}^{d}) \to \hh_g$, and,
with respect to the decomposition $L^2(\mathbb{R}^{2d}) = \hh_g \bigoplus \hh_g^\perp$, the Toeplitz and localization operators are related by
\begin{align}\label{eq_conj}
T^g_\Omega F = P_{\hh_g} 1_\Omega P_{\hh_g} = \begin{bmatrix}
V_g  \A^g  V_g^* F  & 0\\
0 & 0
\end{bmatrix},
\end{align}
see, e.g., \cite{dMFeNo, MR1882695, MR2452833}. Thus, the spectrum of $T_\Omega^g$ and
$\A^g$ coincide except for the multiplicity of the eigenvalue $\lambda=0$.

Suppose that $\Omega\subset\R^{2d}$ has regular boundary at scale $\eta$ with constant $\kappa$ and let us consider different kinds of decay of the window function $g$.

\subsubsection{Gelfand-Shilov window classes}
First, we discuss the case where $g$ belongs to the so-called \emph{Gelfand-Shilov class} $\mathcal{S}^{\beta,\beta}(\mathbb{R}^d)$, i.e., there exist constants $c,C>0$ such that the following decay and smoothness conditions hold:
\begin{align*}
|g(x)| \leq C e^{-c |x|^{1/\beta}}, \qquad |\hat{g}(\xi)| \leq C e^{-c |\xi|^{1/\beta}}, \qquad x,\xi \in \mathbb{R}^d\,.
\end{align*}
Equivalently, 
\begin{align}\label{eq_gsc2}
|V_gg(z)| \leq B e^{-\alpha |z|^{1/\beta}}\,,
\end{align}
for constants $B,\alpha>0$ ---
see \cite{MR3469849, ChChKi}, \cite[Theorem~3.2]{MR1732755} or \cite[Corollary 3.11 and Proposition 3.12]{MR2027858}, which gives the kernel estimate
$|K(z,w)|=|V_gg(z-w)| \leq B e^{-\alpha |z-w|^{1/\beta}}$.

The geometric hypotheses for Theorem~\ref{th:general_bound} and Corollary~\ref{lemma:exp} are easily checked. This is done in a slightly more general setting in Lemma~\ref{lemmix} below, so we omit it here. Applying Lemma~\ref{lemmix}, Corollary~\ref{lemma:exp} and \eqref{eq_gsc2}, it follows that
\begin{align}
\label{eqc13}
        \big|  \#\{\lambda \in \sigma(T^g_{\Omega}):\  \lambda > \delta\} - {   } |\Omega| \big| 
       \lesssim \frac{\HH(\partial \Omega)}{\kappa\, \min\{1,\eta^{2d-1}\}}  \big(\log(\tau )\big)^{ 2d{\beta}} \log^\ast\big(\log(\tau  )\big)\,,
\end{align}
where $\delta\in(0,1)$ and $\tau \coloneqq \max\{\frac{1}{\delta},\frac{1}{1-\delta}\}$. This essentially recovers \cite[Theorem 1.1]{MR}. (The cited result provides a slightly sharper dependence on $\beta$
when $\Omega$ is subject to increasing dilations. To maintain the conciseness of the present paper, we omit a more detailed discussion of this refinement.)

\subsubsection{Polynomial decay}
Second, we turn our attention to windows $g$ such that $|V_g g|$ has polynomial decay. For $s\ge 0$ and $1\le p<\infty$ and a non-zero Schwartz window $\varphi\in\mathcal{S}(\R^d)$, define the modulation space $M^p_s(\R^d)$ as the space of tempered distributions $f\in \mathcal{S}'(\R^d)$ just that
    \begin{align}\label{eqmod}
        \|f\|_{M^p_s}&=\Big(\int_{\R^{2d}} (1+|z|)^{ps} |V_\varphi f(z)|^p\, dz\Big)^{1/p}<\infty\,.
    \end{align}
The choice of the window is not significant since different windows lead to equivalent norms (see 
\cite{benyimodulation} or \cite[Proposition 11.3.2]{grbook}). For $s=0$, we just write $M^p(\R^d)$. 

We now use Theorem~\ref{th:general_bound} together with Lemma~\ref{lemmix} to recover \cite[Theorem 1.4]{MR} (with a slightly more practical formulation in terms of modulation norms).

\begin{proposition}
   Let $g \in L^2(\mathbb{R}^d)$ with $\|g\|_2=1$ and $\Omega\subset\R^{2d}$ a compact set with regular boundary at scale $\eta$ and constant $\kappa$. Consider the operator $T^g_{\Omega}$ defined in \eqref{eq_conj}. If $g\in M_{s}^{4/3}(\R^d)$ for some $s\ge \tfrac{1}{2}$, then, for $\delta\in(0,1)$,
\begin{multline*}
        \big|  \#\{\lambda \in \sigma(T^g_{\Omega}):\  \lambda > \delta\} - {   } |\Omega| \big| \\
       \lesssim 
          \frac{\HH(\partial \Omega)}{\kappa\, \min\{1,\eta^{2d-1}\}}  \cdot  \big(\tau \|g\|_{M_{s}^{4/3}}^4\big)^{\frac{2d}{2s+2d-1}}\cdot \Big(\log^\ast\Big((\tau \|g\|_{M_{s}^{4/3}}^4)^{\tfrac{1}{2s+2d-1}}\Big)\Big)^{\tfrac{2s-1}{2s+2d-1} }\,,
    \end{multline*} 
    where $\tau \coloneqq \max\{\frac{1}{\delta},\frac{1}{1-\delta}\}$.
\end{proposition}
\begin{proof}
    As mentioned, we postpone checking the hypotheses of Theorem~\ref{th:general_bound} to Lemma~\ref{lemmix}, where this is done in a more general setting.

By Lemma~\ref{lemmix} and Theorem~\ref{th:general_bound} applied to $2s\ge 1$,
 \begin{multline*}
        \big|  \#\{\lambda \in \sigma(T^g_{\Omega}):\  \lambda > \delta\} - {   } |\Omega| \big| \notag
       \\ \lesssim  \frac{\HH(\partial \Omega)}{\kappa\, \min\{1,\eta^{2d-1}\}} \big(\tau  N(2s)\big)^{\frac{2d}{2s+2d-1}}\Big(\log^\ast\Big(\big(\tau  N(2s)\big)^{\tfrac{1}{2s+2d-1}}\Big)\Big)^{\tfrac{2s-1}{2s+2d-1} }\,.
    \end{multline*}
     It suffices to estimate the quantity $N(2s)$. We compute the $M^2_s$-norm using the $L^2$-normalized Gaussian window $\varphi(x)=2^{d/4}e^{-\pi|x|^2}$.
        By \cite[Lemma 11.3.3]{grbook} we have $|V_gg|\le |V_\varphi g|\ast |V_g\varphi|$. This together with Young's convolution inequality gives
\begin{align*}
    N(2s)&=\int_{\R^{2d}} (1+|z|)^{2s} |V_gg(z)|^2 \, dz
    \\ &\le \int_{\R^{2d}} \Big( \int_{\R^{2d}} (1+|z-w|)^{s} |V_\varphi g(z-w)| (1+|w|)^{s} |V_g\varphi(w)|\, dw\Big)^2 \, dz  \notag
       \le   \|g\|_{M^{4/3}_{s}}^{4}\,, \notag
\end{align*}
where we used that $|V_g \varphi (z)|=|V_\varphi g(-z)|$.
\end{proof}

 \subsection{Gabor multipliers} 
 
 \label{secgab}
 
With the notation of Section~\ref{sec:application}, let $g \in L^2(\mathbb{R}^d)$ and $\Lambda \subset \mathbb{R}^{2d}$ a full-rank lattice. Also for $(x,\xi)\in \R^{2d}$ denote the \emph{time-frequency shift} operator by
\begin{align*}
\pi(x,\xi) g(t) = g(t-x) e^{2\pi i \xi t}\,\qquad t\in \R^{d}\,.
\end{align*}
We say that the \emph{Gabor system}
\begin{align}
\mathcal{G}(g,\Lambda)
=\{\pi(\lambda) g: \lambda \in \Lambda\}
\end{align}
is a frame of $L^2(\mathbb{R}^d)$ if the \emph{frame operator} $S_g:L^2(\R^d)\to L^2(\R^d)$,
\[ 
S_{g}f=\sum_{\lambda\in\Lambda}\langle f,\pi(\lambda)g\rangle \pi(\lambda)g
\,,\qquad f\in L^2(\R^d)\,,
\]
is invertible on $L^2(\mathbb{R}^d)$. In this case, the function $\dual=S_g^{-1} g$ is known as the \emph{canonical dual window} of $g$ and provides the following expansion: every function $f \in L^2(\mathbb{R}^d)$ can be represented as a norm-convergent series
\begin{align}\label{eq_tf_exp}
f = \sum_{\lambda \in \Lambda} \langle f, \pi(\lambda) g \rangle \pi(\lambda)\dual\,. 
\end{align}
One of the most important results in Gabor theory is that if $g\in M^1(\mathbb{R}^d)$, then $\dual \in M^1(\mathbb{R}^d)$ \cite[Section 4.2]{wiener}, which intuitively means that the coefficients in the expansion \eqref{eq_tf_exp} reflect the time-frequency profile of the function $f$. See \cite[Chapter 5]{grbook} for more background on Gabor frames.

Let $\Omega \subset \mathbb{R}^{2d}$ and consider the \emph{Gabor multiplier}
\begin{align}\label{eq_gm}
M_{g,\Lambda,\Omega} f = \sum_{\lambda \in \Lambda \cap \Omega} \langle f, \pi(\lambda) g \rangle \pi(\lambda)\dual\,,
\end{align}
which is an approximation of the concentration operator \eqref{eq_tf_dau}. 
We are interested in spectral properties of $M_{g,\Lambda,\Omega}$ that are uniform with respect to the lattice $\Lambda$. 

\begin{remark}
    \label{rem_gm}
    The quantities involved in Theorem~\ref{th:frame_conc} simplify in the case of Gabor multipliers. Indeed, if we invoke the so-called Ron-Shen duality (also known as Wexler-Raz relations, see \cite[Theorem 7.3.1]{grbook}) we see that 
$\langle g, \dual \rangle = |\Lambda|\,.
$ 
This shows that   $\Mg=\mg=|\Lambda|$ as well as $\mu_g=\mu_\Lambda$. In particular,  $C_\varphi=1,$
    \[N(s)=|\Lambda| \sum_{\lambda\in\Lambda}(1+|\lambda|)^s\big|\big\langle \tfrac{\dual}{|\Lambda|},\pi(\lambda)g\big\rangle\big|^2\,,\qquad \text{and}\qquad
D= |\Lambda| \sum_{\lambda \in \Lambda} e^{\alpha |\lambda|^{1/\beta}} \big|\big\langle \tfrac{\dual}{|\Lambda|},\pi(\lambda)g\big\rangle\big|^2
       \,.\]     
       Moreover, as in Remark~\ref{rem_tight} one may apply \eqref{eq:comp-Omega-Omega_Lambda} to measure the spectral deviation from the Lebesgue measure of $\Omega$ instead of $\mu_\Lambda(\Omega)$:\begin{equation*}
        \big|  \#\{\lambda \in \sigma(M_{g,\Lambda,\Omega}):\  \lambda > \delta\} - |\Omega| \big|\leq  \big|  \#\{\lambda \in \sigma(M_{g,\Lambda,\Omega}):\  \lambda > \delta\} - {   } \mu_\Lambda(\Omega_\Lambda) \big|+\big|\mu_\Lambda(\Omega_\Lambda)-|\Omega|\big|.
         \end{equation*}  
\end{remark}

Note that, up to this point, the quantities $N(s)$ and $D$ depend on both $g$ and its canonical dual window. In what follows, we show that for sufficiently fine discretizations of the phase space, these conditions can be simplified so that they depend only on $g$, making the result essentially compatible with the continuous setting.

\begin{theorem}\label{cor:gabor}
Let $g \in L^2(\mathbb{R}^d)$ with $\|g\|_2=1$, $\Lambda = A\Z^{2d}$ a lattice with $A\in \gl_{2d}(\R)$ and $\Omega\subset\R^{2d}$ a compact set with regular boundary at scale $\eta$ and constant $\kappa$. Let $\|A\|$ be the spectral norm of $A$ and $\vr$ its condition number. Suppose that the Gabor system $\mathcal{G}(g,\Lambda)$ is a frame of $L^2(\mathbb{R}^d)$ and consider the Gabor multiplier $M_{g,\Lambda,\Omega}$ defined in \eqref{eq_gm}. The following statements hold:
\begin{enumerate}[label=(\roman*)] 
\item\label{cgi} Suppose that $g\in M_{s}^1(\R^d)$ for some $s\ge \tfrac{1}{2}$. If $\|A\|<\sigma$ for a sufficiently small $0<\sigma=\sigma(g,s,d)\le 1$, then
\begin{align*}
        \big|  \#\{\lambda &\in \sigma(M_{g,\Lambda,\Omega}):\  \lambda > \delta\} - {   } |\Omega| \big| \notag
       \lesssim \begin{multlined}[t]
           \vr^{4d}
       \cdot\frac{\HH(\partial \Omega)}{\kappa\, \min\{1,\eta^{2d-1}\}} \cdot s^d \log^\ast(s) \\ \cdot  \big(\tau \|g\|_{M_{s}^1}^4\big)^{\frac{2d}{2s+2d-1}}\Big(\log^\ast\Big((\tau \|g\|_{M_{s}^1}^4)^{\tfrac{1}{2s+2d-1}}\Big)\Big)^{\tfrac{2s-1}{2s+2d-1} }\,,
       \end{multlined}
    \end{align*}
\item\label{cgii} Suppose that $g$ satisfies the Gelfand-Shilov condition \eqref{eq_gsc2} with parameters $\beta\ge 1$ and $\alpha,B>0$. If $\|A\|<\sigma$ for a sufficiently small $0<\sigma=\sigma(B,\alpha,\beta,d)\le 1$, then 
\begin{align*}
        \big|  \#\{\lambda \in \sigma(M_{g,\Lambda,\Omega}):\  \lambda > \delta\} - {   } |\Omega| \big| \notag
       \lesssim \vr^{4d}\cdot \frac{\HH(\partial \Omega)}{\kappa\, \min\{1,\eta^{2d-1}\}} \cdot  \big(\log^\ast(B\tau )\big)^{ 2d{\beta}} \log^\ast\big(\log^\ast(B\tau) \big)\,.
    \end{align*}
\end{enumerate}  
\end{theorem}
\begin{proof}
    First, let us mention that the frame condition is not a real imposition since it is satisfied for sufficiently small $\sigma$ by Janssen's criterion. Regarding \ref{cgi}, apply Theorem~\ref{th:frame_conc} \ref{tfi} and Remark~\ref{rem_gm} with parameter $2s\ge 1$. It suffices to estimate the quantity $N(2s)$. We compute the modulation space norm \eqref{eqmod} using the $L^2$-normalized Gaussian window $\varphi(x)=2^{d/4}e^{-\pi|x|^2}$ and for $1\le p<\infty$, $f\in \mathcal{S}'(\R^d)$ we also write
    \begin{align*}
         \|f\|_{M^p_{s,\Lambda}}=\Big( \sup_{z\in AQ} |\Lambda| \sum_{\lambda\in\Lambda} (1+|\lambda + z|)^{ps}|V_\varphi f(\lambda +z)|^{p}\Big)^{1/p}\,,
    \end{align*}
    where $Q=[-\tfrac{1}{2},\tfrac{1}{2}]^{2d}$. 
    By \cite[Lemma 11.3.3]{grbook} we have $|V_gf|\le |V_\varphi f|\ast |V_g\varphi|\ast|V_\varphi \varphi|$. This together with Cauchy-Schwarz inequality gives
\begin{align}\label{eqconv}
    N(2s)&\hspace{-1pt}=\hspace{-1pt}|\Lambda| \sum_{\lambda\in\Lambda}(1+|\lambda|)^{2s}|V_g \tfrac{\dual}{|\Lambda|}(\lambda)|^2
    \le |\Lambda| \sum_{\lambda\in\Lambda}(1+|\lambda|)^{2s}\big(\big|V_\varphi \tfrac{\dual}{|\Lambda|}\big|\ast|V_g \varphi|\ast|V_\varphi \varphi|(\lambda)\big)^2
    \\ &\hspace{-1pt}\le\hspace{-1pt} |\Lambda| \sum_{\lambda\in\Lambda}\left[\Big((1+|\cdot|)^{s}\big(\big|V_\varphi \tfrac{\dual}{|\Lambda|}\big|\ast|V_g \varphi|\big)\Big)
    \ast\Big((1+|\cdot|)^{s}|V_\varphi \varphi|\Big)(\lambda)\right]^2 \notag
    \\ &\hspace{-1pt}\le \hspace{-1pt}\big\||\Lambda|^{-1}\dual\big\|_{M^{1}_{s}} \|g\|_{M^{1}_{s}} |\Lambda| \sum_{\lambda\in\Lambda}\int_{\R^{2d}}(1\hspace{-.5pt}+\hspace{-.5pt}|z|)^{s}\big|V_\varphi \tfrac{\dual}{|\Lambda|}\big|\ast |V_g \varphi|(z) (1\hspace{-.5pt}+\hspace{-.5pt}|\lambda\hspace{-.5pt} -\hspace{-.5pt}z|)^{2s}|V_\varphi \varphi(\lambda\hspace{-.5pt}-\hspace{-.5pt}z)|^2\, dz \notag
    \\ &
  \hspace{-1pt}  \le \hspace{-1pt}\|\varphi\|_{M^2_{s,\Lambda}}^2 \big\||\Lambda|^{-1}\dual\big\|_{M^{1}_{s}}^2 \|g\|_{M^{1}_{s}}^{2}\,, \notag
\end{align}
where we used that $|V_g \varphi (z)|=|V_\varphi g(-z)|$.

Since $V_\varphi \varphi(z)=e^{-\tfrac{\pi}{2}|z|^2}$ (see \cite[Lemma 1.5.2]{grbook}), a straightforward computation shows that if $\|A\|$ is smaller than an absolute constant, then there exists a constant $C_d>0$ such that
\begin{align}
    \label{eqga}
    \|\varphi\|_{M^2_{s,\Lambda}}^2\lesssim (C_d)^ss^{s+d}\,.
\end{align}
It remains to estimate the norm of $\dual$ in terms of $g$. Define the adjoint lattice associated to $\Lambda$ by
\[\Lambda^\circ=\{\lambda^\circ\in\R^{2d}: \ \pi(\lambda^\circ)\pi(\lambda)=\pi(\lambda)\pi(\lambda^\circ), \text{ for every }\lambda\in\Lambda\}=JA^{-t}\Z^{2d}\,,\]
where
\[J=\begin{pmatrix}
    0 & -\id_d
\\ \id_d & 0
\end{pmatrix}\,.\]
As shown in \cite[Corollary 9.4.5]{grbook}, \begin{align*}
    |\Lambda|S_g-\id=\sum_{\lambda^\circ\in\Lambda^\circ\smallsetminus\{0\}} V_g g(\lambda^\circ) \pi(\lambda^\circ)\,.
\end{align*}
Now notice that by \cite[Lemma 11.1.2]{grbook} and proceeding as in \eqref{eqconv},
\begin{align}\label{eqneu}
   \big\||\Lambda|S_g-\id\big\|_{M^{1}_{s}\to M^{1}_{s}}&\le \sum_{\lambda^\circ\in\Lambda^\circ\smallsetminus\{0\}} |V_g g(\lambda^\circ)| 
\|\pi(\lambda^\circ)\|_{M^{1}_{s}\to M^{1}_{s}}
\\ &
\lesssim \sum_{\lambda^\circ\in\Lambda^\circ\smallsetminus\{0\}} (1+|\lambda^\circ|)^{s}|V_g g(\lambda^\circ)| \notag
\\ &
\lesssim \sum_{\lambda^\circ\in\Lambda^\circ\smallsetminus\{0\}} 
    \int_{\R^{2d}}\int_{\R^{2d}}
(1 + |z|)^{s}|V_\varphi g|(z) 
(1+|w|)^{s}|V_\varphi g|(-w) \notag
\\ &\hspace{3cm}\cdot
(1 + |\lambda^\circ\hspace{-.5pt} - z- w|)^{s}|V_\varphi \varphi(\lambda^\circ - z- w)|\, dz dw\,.
 \notag
\end{align}
For $r>0$, we decompose both integrals
 by splitting each variable's domain into $B_r(0)$ and $B_r(0)^c$. This yields
 \begin{align}
 \big\||\Lambda|S_g-\id\big\|_{M^{1}_{s}\to M^{1}_{s}}
 &\lesssim 
  \|g\|_{M^1_s}
 \int_{B_r(0)^c}
(1\hspace{-.5pt}+\hspace{-.5pt}|z|)^{s}|V_\varphi g|(z) 
\, dz \sup_{z\in\R^{2d}}\sum_{\lambda^\circ\in\Lambda^\circ} (1+|\lambda^\circ-z|)^{s}|V_\varphi \varphi(\lambda^\circ-z)| \notag
\\ &\hspace{3.3cm}+ \|g\|_{M^{1}_{s}}^{2} \sup_{z\in B_{2r}(0)} \sum_{\lambda^\circ\in\Lambda^\circ\smallsetminus\{0\}} (1+|\lambda^\circ-z|)^{s}|V_\varphi \varphi(\lambda^\circ-z)| 
 \notag
\\  &\lesssim 
  C_{d,s}\|g\|_{M^1_s}
 \int_{B_r(0)^c}
(1\hspace{-.5pt}+\hspace{-.5pt}|z|)^{s}|V_\varphi g|(z) 
\, dz \notag
\\& \hspace{3.3cm}+ \|g\|_{M^{1}_{s}}^{2} \sup_{z\in B_{2r}(0)} \sum_{\lambda^\circ\in\Lambda^\circ\smallsetminus\{0\}} (1+|\lambda^\circ-z|)^{s}e^{-\tfrac{\pi}{2}|\lambda^\circ-z|^2} 
<\frac{1}{2} \notag
\,,
\end{align}
where $C_{d,s}>0$ is a constant depending on $d$ and $s$ and for the last step we assume $r$ is sufficiently big and $\|A\|$ is sufficiently small. In particular,
\[\big\||\Lambda|^{-1}\dual\big\|_{M^{1}_{s}}=\big\|(|\Lambda|S_g)^{-1} g\big\|_{M^{1}_{s}}\le
\big\|(|\Lambda|S_g)^{-1}\big\|_{M^{1}_{s}\to M^{1}_{s}}\|g\|_{M^{1}_{s}}\le 2 \| g\|_{M^{1}_{s}} \,.\] 
Joining this with \eqref{eqconv} and \eqref{eqga}, yields
\[N(2s)\lesssim  C_d^ss^{s+d} \|g\|_{M^{1}_{s}}^4\,.\]
Plugging this into Theorem~\ref{th:frame_conc} gives \ref{cgi}.

The proof of \ref{cgii} is quite similar, so we comment on the overall argument and skip the details. We can work directly with $g$ rather than $\varphi$ and proceed as in $\eqref{eqconv}$ to get
\begin{align*}
    D=|\Lambda| \sum_{\lambda \in \Lambda} e^{\alpha |\lambda|^{1/\beta}} \big|\big\langle \tfrac{\dual}{|\Lambda|},\pi(\lambda)g\big\rangle\big|^2
&\lesssim  |\Lambda| \sum_{\lambda\in\Lambda}\left[\Big(e^{\tfrac{\alpha}{2} |\cdot|^{1/\beta}}|V_g \tfrac{\dual}{|\Lambda|}|
   \Big) \ast \Big( e^{\tfrac{\alpha}{2} |\cdot|^{1/\beta}}|V_g g|\Big)(\lambda)\right]^2
    \\ & \lesssim \Big\|e^{\tfrac{\alpha}{2} |\cdot|^{1/\beta}}V_g \tfrac{\dual}{|\Lambda|}\Big\|_1^2 
     \sup_{z\in AQ} |\Lambda| \sum_{\lambda\in\Lambda}  B^2  e^{-\alpha |\lambda-z|^{1/\beta}},
    \\ & \lesssim B^2 \Big\|e^{\tfrac{\alpha}{2} |\cdot|^{1/\beta}}V_g \tfrac{\dual}{|\Lambda|}\Big\|_1^2\,, 
\end{align*}
where in the last step we assume $\|A\|$ is less than a sufficiently small absolute constant and the implied constant depends on $d,\alpha,\beta$. Define
\[\|f\|_{\alpha,\beta}=\big\|e^{\tfrac{\alpha}{2} |z|^{1/\beta}} V_g f\big\|_1\,, \qquad f\in \mathcal{S}'(\R^d)\,.\]
Similarly to \eqref{eqneu}, by \cite[Lemma 11.1.2]{grbook} (notice that for this we need $\beta\ge 1$) and \eqref{eq_gsc2},
\begin{align*}
    \big\|(|\Lambda|S_g-\id) f\big\|_{\alpha,\beta}&\lesssim \| f\|_{\alpha,\beta} \sum_{\lambda^\circ\in\Lambda^\circ\smallsetminus\{0\}} e^{\tfrac{\alpha}{2} |\lambda^\circ|^{1/\beta}} |V_g g(\lambda^\circ)|  
    \\ &\lesssim B\| f\|_{\alpha,\beta} \sum_{\lambda^\circ\in\Lambda^\circ\smallsetminus\{0\}} e^{-\tfrac{\alpha}{2} |\lambda^\circ|^{1/\beta}}   
    <\frac{1}{2}\| f\|_{\alpha,\beta}\,,
\end{align*}
provided $\|A\|$ is sufficiently small. The result now follows from Theorem~\ref{th:frame_conc} and Remark~\ref{rem_gm} as before.
\end{proof}

\begin{proof}[Proof of Corollary \ref{cor_gm_prel}]
    The results follows from Theorem~\ref{cor:gabor} and the fact that in dimension 2 connectedness implies Ahlfors regularity with parameters $1\le \kappa\le 2$ and $\eta=\HH(\partial\Omega)$ (e.g. \cite[Lemma~2.5]{MR2}).
\end{proof}

\section{More applications}
\label{sec_moreappl}
\subsection{Mixed-state localization operators}\label{sec_mixed}

The short-time Fourier transform of a function $f\in L^2(\R^d)$ using an operator window $S\in B(L^2(\R^d))$, given by  
$$
\V_Sf(z)=S^\ast\pi(z)^\ast f\,,\qquad z\in\R^{2d}\,,
$$
was first introduced in \cite{eirik}. Note that this transform is vector-valued, i.e., $\V_Sf(z)\in L^2(\R^d)$ for every $z\in \R^{2d}$.  It turns out that this object shares many properties with the classical short-time Fourier transform, see \cite{operator-stft,eirik}.
In particular, Moyal's identity
\begin{equation}\label{eq:moyal}
\int_{\R^{2d}}\big\langle\V_Sf_1(z),\V_Sf_2(z)\big\rangle \,dz=\langle f_1,f_2\rangle \|S\|_{\S_2}^2\,,
\end{equation}
holds for every $f_1,f_2\in L^2(\R^d)$ and $S\in\S_2$. Let us from now on assume that $\|S\|_{\S_2}=1$. In that case, \eqref{eq:moyal} implies that 
$\hh_S:=\V_S(L^2(\R^d))$ is a reproducing kernel subspace of $L^2(\R^{2d},L^2(\R^d))$ with an operator-valued reproducing kernel
$$
K(z,w)=S^\ast\pi(z)^\ast\pi(w)S\,,
$$
and its associated orthogonal projection $P_S\colon L^2(\R^{2d},L^2(\R^d)) \to \hh_S$ is given by $P_S=\V_S\V_S^*$.
Using the convention $f\otimes g:=\langle\, \cdot\, ,g\rangle f$ to denote the tensor product of two functions, we express $S$ by its singular value decomposition 
 $S=\sum_{n\in\N}\nu_n (g_n\otimes h_n)$, where $\{g_n\}_{n\in\N}$, $\{h_n\}_{n\in\N}$  are two orthonormal families in $L^2(\R^d)$, and  $\sum_n|\nu_n|^2=1$. Note that if $S$ has rank one we get $\V_Sf(z)= \langle f,\pi(z)g\rangle h$, and $K(z,w)=\langle \pi(w)g,\pi(z)g\rangle (h\otimes h)$. In other words, we essentially recover the scalar case from Section~\ref{sec-stft}.

For $m\in L^1(\R^{2d})$ and $R\in \S_1$, define their function-operator convolution $m\star R\in \S_1$ via
$$
m\star R:=\int_{\R^{2d}}m(z)\pi(z)R\pi(z)^\ast \,dz\,.
$$
For a compact set $\Omega\subseteq\R^{2d}$, define the \emph{mixed-state localization operator} $\A^S:L^2(\R^d)\to L^2(\R^d)$ as
\[\A^S=1_\Omega \star( SS^\ast)\,.\]
Using the spectral decomposition of $SS^\ast$, one may write out $\A^S$ explicitly as
$$
\A^S=\V_S^\ast  1_\Omega \V_S=\sum_{n\in\N} |\nu_n|^2 \A^{g_n}\,,
$$
where $\A^g$ is as in \eqref{eq_tf_dau}. In particular, a mixed-state localization operator is a weighted sum of standard localization operators.

Just as in the scalar case, the operator $T_\Omega=P_S1_\Omega P_S$  has the same non-zero eigenvalues as $A_\Omega^S$ --- see \cite[Section 4.3]{operator-stft}. The following lemma allows us to apply Theorem~\ref{th:general_bound} and Corollary~\ref{lemma:exp} to $T_\Omega=P_S 1_\Omega P_S$, and consequently to $A_\Omega^S$.  

\begin{lemma}\label{lemmix}
    Consider $\R^{2d}$ endowed with the Euclidean distance and the Lebesgue measure $\mu$. Suppose $\Omega\subset\R^{2d}$ is a compact set with regular boundary at scale $\eta$ and constant $\kappa$. Then, $\R^{2d}$ is a locally compact convex Poincar\'e space and conditions \ref{norm2}-\ref{doubling}  are satisfied. Moreover,
    \begin{align*}
        &(i) & \gamma&=2d\,;
        \\ &(ii) & C_{\R^{2d}}&=2^{2d}\,;
        \\ &(iii) & \max\{\ii(\Omega),\pp(\Omega)\}&\lesssim  \frac{\HH(\partial \Omega)}{\kappa\, \min\{1,\eta^{2d-1}\}}\,;
         \\ &(iv)& N_{\hh_S}(s)&=\int_{\R^{2d}}(1+|z|)^s\sum_{n,m\in\N}|\nu_n|^2|\nu_m|^2|\langle g_n,\pi(z)g_m\rangle|^2\,dz\,;
         \\ &(v)& D_{\hh_S}&=\int_{\R^{2d}} e^{\alpha |z|^{1/\beta}}\sum_{n,m\in\N}|\nu_n|^2|\nu_m|^2|\langle g_n,\pi(z)g_m\rangle|^2\,dz\,.
    \end{align*}
\end{lemma}

Note that Theorem~\ref{th:general_bound} combined with the parameters as above greatly improves the estimates of \cite[Lemma~4.3]{accumulated-cohen} for $\text{trace}\big(A_\Omega^S-(A_\Omega^S)^2\big)$ that correspond to \eqref{eq:last_common2} when $s=1$ (or more precisely, to \eqref{eq:H} for $p=2$ and $\alpha=1$). On the other hand, if $S=g\otimes h$, then $\A^S=\A^g$ from \eqref{eq_tf_dau}. So, Lemma~\ref{lemmix} applies to the scalar case as mentioned in Section~\ref{sec-stft}.

\begin{proof}[Proof of Lemma~\ref{lemmix}]
    From the singular value decomposition of $S$, we can explicitly compute the kernel
\begin{align*}
K(z,w)=\sum_{n,m\in\N}\nu_n\nu_m\langle \pi(w)g_n,\pi(z)g_m\rangle (h_m\otimes h_n)\,,
\end{align*}
from where it follows that,
\begin{equation}\label{eq:kernel-norm}
\|K(z,w)\|_{\S_2}^2=\sum_{n,m\in\N}|\nu_n|^2|\nu_m|^2|\langle \pi(w)g_n,\pi(z)g_m\rangle|^2\,.
\end{equation}
Therefore, by Moyal's identity for the short-time Fourier transform
$$
\int_{\R^{2d}}\|K(z,w)\|_{\S_2}^2\,dz =\sum_{n,m\in\N}|\nu_n|^2|\nu_m|^2 \|g_n\|_2^2 \|g_m\|_2^2=1\,,
$$
that is, \ref{norm2} holds. 

Condition~\ref{doubling} is clearly met for the Lebesgue measure on $\R^{2d}$ with doubling constant $C_{\R^{2d}}=2^{2d}$. Regarding \ref{stripe}, note that $\R^{2d}$ is a complete doubling Poincar\'e space (see \cite{BB}).
In particular, if $\Omega$ has a regular boundary, then Proposition~\ref{peri} and Proposition~\ref{reduct} provide the stated estimates for $\pp(\Omega)$, $\ii(\Omega)$ and $\gamma$.
Finally, from the group structure of $\R^{2d}$ it follows that \eqref{eqn} and \eqref{eq:exp} simplify to
\begin{align*}
    N_{\hh_S}(s)&=\int_{\R^{2d}}(1+|z|)^s\sum_{n,m\in\N}|\nu_n|^2|\nu_m|^2|\langle g_n,\pi(z)g_m\rangle|^2\,dz\,,
\intertext{and}
    D_{\hh_S}&= \int_{\R^{2d}} e^{\alpha |z|^{1/\beta}}\sum_{n,m\in\N}|\nu_n|^2|\nu_m|^2|\langle g_n,\pi(z)g_m\rangle|^2\,dz\,.\qedhere
\end{align*}

\end{proof}

\subsection{Fourier concentration operators}\label{sec:fou}
Finally, we discuss how our results on kernels with fast off-diagonal decay can potentially be used effectively on slowly decaying kernels. As a proof of concept, we shall reinterpret \cite{israel15} as a decomposition method, 
which enables the application of Theorem~\ref{th:general_bound} to the space of bandlimited functions.

Let $I,J \subset \mathbb{R}$ be compact intervals and let us normalize the Fourier transform by $\F {f}(\xi)=\int_{\mathbb{R}} f(x) e^{-2\pi i \xi x}\,dx$. Let $\hh$ be the \emph{Paley-Wiener space} of square integrable functions with Fourier transform supported on $I$, with orthogonal projection 
$P_I=\F^{-1}1_I\F$.

The \emph{Fourier concentration operator} is defined by applying a spatial cut-off, followed by a frequency cut-off, $1_I P_J 1_I$, and has the same non-zero eigenvalues (including multiplicities) as the Toeplitz operator
\begin{align}
    \label{eqtef}
    T_{I,J}=P_I 1_J P_I\,.
\end{align} 
The operator $T_{I,J}$ was originally studied in the seminal papers \cite{SlPo,LaPo2,LaPo3} and has proved its usefulness for a variety of applications (see \cite{BoJaKa,Th,ZeMe,DaWa}). 
One-parameter spectral asymptotics for $T_{I,J}$ go back to \cite{La2,LaWi} --- see also
\cite{Wi,So} --- while the study of two-parameter spectral asymptotics is much more recent \cite{KaRoDa,Os,israel15,BoJaKa}. 
The higher dimensional case, with general domains in lieu of intervals  was treated in \cite{IsMa,eigenfourier,HuIsMa,HuIsMa2,kulikov2026sharp}. 

For a compact self-adjoint operator $0\le T \le 1$, define
\begin{align}
    \label{eqpl}
    M_\delta(T)=\#\{\lambda\in\sigma(T):\delta<\lambda<1-\delta\}\,,\qquad \delta\in(0,\tfrac{1}{2})\,.
\end{align}
which measures the number of intermediate eigenvalues. For concentration operators, this is referred to as the size of the \emph{transition} or \emph{plunge region} and essentially encodes the error of the approximation $\{\lambda\in\sigma(T):\lambda>\delta\}\sim \trace(T)$.

We now revisit the following result from \cite{israel15} and reinterpret its proof.
\begin{theorem}[{\cite[Theorem 2]{israel15}}]\label{theois}
    For $\beta>1$, there exists a constant $C_\beta>0$ such that, for every pair of intervals $I,J\subseteq\R$ and $T_{I,J}$ as in \eqref{eqtef}, 
    \[M_\delta(T_{I,J})\le C_\beta\cdot \log^*(|I||J|\delta^{-1})  \cdot\log^*\big(\delta^{-1}\log^*(|I||J|)\big)^\beta\, ,\qquad \delta\in(0,\tfrac{1}{2})\, .\]
\end{theorem}
\begin{proof}[Sketch of the proof, based on \cite{israel15}]
\noindent {\bf Step 1} \emph{(Wave-packet decomposition)}. Since rescaling time by a factor $r>0$ rescales the frequency variable by $r^{-1}$, we can assume without loss of generality that $|J|=1$. 

As in \cite{israel15}, we shall use a Coifman-Meyer decomposition of $L^2(I)$ from \cite{CoMe}, which translates to a decomposition of $\hh$ via the Fourier transform. Specifically, let $\{x_n\}_{n\in\Z}\subseteq I$  with $x_n<x_{n+1}$ such that $I=\bigcup_n I_n$ where $I_n=[x_n,x_{n+1})$ form a Whitney decomposition of $I$: 
\[|I_n|\le  \mathrm{d}(I_n,I^c)\le 5|I_n|\,.\]
From \cite[Section~3]{israel15} or 
\cite[Chapter 1.3 and Equation (3.10)]{MR1408902}, given $\beta>1$, there exist functions $\theta_n\in\mathcal{C}^\infty(\R)$ with the following properties.
	\begin{enumerate}[label=(\roman*)]
		\item\label{theti} $0\le\theta_n\le 1$, $\theta_n(x)=1$ for $x\in \big[x_n+\tfrac{\delta_n}{10},x_{n+1}-\tfrac{\delta_n}{10}\big]$, $\theta_n(x)=0$ for $x\notin \big[x_n-\tfrac{\delta_n}{10},x_{n+1}+\tfrac{\delta_n}{10}\big]$, where $\delta_n=x_{n+1}-x_n$.
		\item\label{thetii} There are constants $c_\beta,C_\beta>0$ independent of $n$ such that $|\widehat{\theta_n}(x)|\le C_\beta e^{-c_\beta |\delta_n x|^{1/\beta}}$.
		\item\label{thetiii} The maps $Q_n:L^2(\R)\to L^2(\R)$, given by
        \begin{align*}
            Q_nf(x)= \theta_n(x)\big(\theta_n(x) f(x)+\theta_n(2x_n-x) f(2x_n-x)+\theta_n(2x_{n+1}-x) f(2x_{n+1}-x)\big)\,,
        \end{align*}
        are orthogonal projections onto orthogonal subspaces $W_n\subseteq L^2\big[x_n-\tfrac{\delta_n}{10},x_{n+1}+\tfrac{\delta_n}{10}\big)$. Moreover, one has $L^2(I)=\bigoplus_n W_n$.
	\end{enumerate}

\smallskip
    
\noindent {\bf Step 2} \emph{(Reduction to orthogonal components)}.
Given $\eta\in(0,\tfrac{1}{2})$, we write
\[\hh =\bigoplus_n \hh_n=\Big(\bigoplus_{n\in A} \hh_n\Big)\oplus \widetilde\hh\,,\]
where $\hh_n=\F^{-1}W_n$ and $A$ is the set of indices for which $ \mathrm{d}(I_n,I^c)>\eta$.  
Denote the orthogonal projections onto $\hh_n$ and $\widetilde \hh$ by $P_n$ and $\widetilde P$ respectively and write
\[ T_n=P_n 1_J P_n\,,\qquad S_n=1_J P_n 1_J \,,\qquad  \widetilde T=\widetilde P 1_J \widetilde P\,, \qquad \text{and} \qquad \widetilde S= 1_J  \widetilde P  1_J \,.\]
Notice that $T_{n}$ and $S_n$ as well as $\widetilde T$ and $\widetilde S$ share the same non-zero eigenvalues. For $\delta\in (0,\tfrac{1}{2})$, a straightforward computation applying the Courant-Fischer formula to the operator $S-S^2$ where $S=\widetilde S+\sum_n S_n$ gives
    \begin{align}
        \label{eqdec}
    M_\delta(T_{I,J})\le M_{\delta_A}(\widetilde T)+\sum_{n\in A} M_{\delta_A}(T_{n}) \,,
    \end{align}
where $\delta_A=\big(\tfrac{\delta}{2(\#A+1)}\big)^2$.
\smallskip

\noindent {\bf Step 3} \emph{(Reproducing kernel estimates)}.
It is easy to check that $\|\widetilde T\|\le 6\eta$ and so
\begin{align}
\label{eqmrest}
    M_{6\eta}(\widetilde T)=0
    \,.
\end{align}
Regarding $T_n$, let $U_n:L^2(\R)\to L^2(\R)$ be the rescaling given by
$U_nf(x)=\sqrt{\delta_n}f(\delta_n x)\,.$
Then, $T_n$ shares the same eigenvalues with
\[U_n^{-1}T_n U_n= U_n^{-1}P_n U_n 1_{\delta_n J} U_n^{-1}P_n U_n=:\Pi_n 1_{\delta_n J} \Pi_n\,.\]
Let $K_n(x,y)$ be the reproducing kernel of $U_n^{-1}\hh_n$. One can check that 
\[1/2\le \|(K_n)_x\|_2^2\le 9\,,\qquad \text{and}\qquad |K_n(x,y)|\le C_\beta' e^{-c_\beta' |\delta_n (x-y)|^{1/\beta}} \,,\]
where $c_\beta',C_\beta'$ do not depend on $\delta_n$. 
Notice that one can renormalize the measure to satisfy \ref{norm2}, and \ref{stripe} clearly holds with $\gamma=1$ and $\ii(\Omega)=4$. 
Applying Corollary~\ref{lemma:exp} to $\Pi_n 1_{\delta_n J} \Pi_n$ (and slightly increasing $\beta$ to avoid $\log\log$ terms), we get for a constant $D_\beta>0$ independent of $n$,
\begin{align}
\label{eqmn}
    M_\eta(T_n)\le D_\beta \log^\ast(\eta^{-1})^\beta \,.
\end{align}
Take 
\[\eta= \Big(\frac{c\delta}{\log^*(|I|\delta^{-1})}\Big)^2\,,\]
for some sufficiently small constant $c>0$.
From  \eqref{eqdec}, \eqref{eqmrest} and \eqref{eqmn}, we conclude that
\[M_\delta\le M_{6\eta}(\widetilde T)+\sum_{n\in A} M_{6\eta}(T_n)
\lesssim \#A  \log(\eta^{-1})^\beta \lesssim \log^*(|I|\delta^{-1})  \log^*\big(\delta^{-1}\log^*(|I|) \big)^\beta\,.\qedhere\]
\end{proof}

\section*{Acknowledgements}
This research was funded by the Austrian Science Fund (FWF)  10.55776/Y1199 (J.L.R., M.S. and L.V.)  and 10.55776/PAT1384824 (F.M. and M.S.).  F.M. was also supported by the EPSRC UKRI/EP/C003286/1.
L. V. acknowledges support from the Erasmus+ Scholarship for Traineeship, issued by the Università di Genova; number 2023/93.  This work elaborates in part on the master's thesis \cite{valtesis}. For open
access purposes, the authors have applied a CC BY public copyright license to any author-accepted manuscript
version arising from this submission.

\bibliographystyle{abbrv}
\bibliography{biblio}

\begin{thebibliography}{10}

\bibitem{AmMiPa}
L.~Ambrosio, M.~Miranda~Jr., and D.~Pallara.
\newblock Special functions of bounded variation in doubling metric measure spaces.
\newblock In {\em Calculus of {V}ariations: {T}opics from the {M}athematical {H}eritage of {E}. {D}e {G}iorgi}, volume~14 of {\em Quad. Mat.}, pages 1--45. Dept. Math., Seconda Univ. Napoli, Caserta, 2004.

\bibitem{AMR25}
Y.~Ameur, F.~Marceca, and J.~L. Romero.
\newblock Gaussian beta ensembles: the perfect freezing transition and its characterization in terms of {B}eurling-{L}andau densities.
\newblock {\em Ann. Inst. H. Poincar{\'e} Probab. Statist.}, 62(1):296--327, 2026.

\bibitem{AR23}
Y.~Ameur and J.~L. Romero.
\newblock The planar low temperature {C}oulomb gas: separation and equidistribution.
\newblock {\em Rev. Mat. Iberoam.}, 39(2):611–648, 2023.

\bibitem{benyimodulation}
{\'A}.~B{\'e}nyi and K.~A. Okoudjou.
\newblock {\em Modulation Spaces: With Applications to Pseudodifferential Operators and Nonlinear Schr{\"o}dinger Equations}.
\newblock Applied and Numerical Harmonic Analysis. Birkh{\"a}user Basel, 2020.

\bibitem{BB}
A.~Bj\"{o}rn and J.~Bj\"{o}rn.
\newblock {\em Nonlinear {P}otential {T}heory on {M}etric {S}paces}, volume~17 of {\em EMS Tracts in Mathematics}.
\newblock European Mathematical Society (EMS), Z\"{u}rich, 2011.

\bibitem{BoGo}
S.~G. Bobkov and F.~G\"{o}tze.
\newblock Discrete isoperimetric and {P}oincar\'{e}-type inequalities.
\newblock {\em Probab. Theory Related Fields}, 114(2):245--277, 1999.

\bibitem{BoJaKa}
A.~Bonami, P.~Jaming, and A.~Karoui.
\newblock Non-asymptotic behavior of the spectrum of the sinc-kernel operator and related applications.
\newblock {\em J. Math. Phys.}, 62(3):Paper No. 033511, 2021.

\bibitem{Ca}
D.~G. Caraballo.
\newblock Areas of level sets of distance functions induced by asymmetric norms.
\newblock {\em Pacific J. Math.}, 218(1):37--52, 2005.

\bibitem{MR1732755}
J.~Cho.
\newblock A characterization of {G}elfand-{S}hilov space based on {W}igner distribution.
\newblock {\em Commun. Korean Math. Soc.}, 14(4):761--767, 1999.

\bibitem{ChChKi}
J.~Chung, S.-Y. Chung, and D.~Kim.
\newblock Characterizations of the {G}elfand-{S}hilov spaces via {F}ourier transforms.
\newblock {\em Proc. Amer. Math. Soc.}, 124(7):2101--2108, 1996.

\bibitem{MR1882695}
L.~A. Coburn.
\newblock The {B}argmann isometry and {G}abor-{D}aubechies wavelet localization operators.
\newblock In {\em Systems, approximation, singular integral operators, and related topics ({B}ordeaux, 2000)}, volume 129 of {\em Oper. Theory Adv. Appl.}, pages 169--178. Birkh{\"a}user, Basel, 2001.

\bibitem{CoMe}
R.~R. Coifman and Y.~Meyer.
\newblock Remarques sur l'analyse de {F}ourier \`a fen\^{e}tre.
\newblock {\em C. R. Acad. Sci. Paris S\'{e}r. I Math.}, 312(3):259--261, 1991.

\bibitem{da88}
I.~Daubechies.
\newblock Time-frequency localization operators: a geometric phase space approach.
\newblock {\em IEEE Trans. Inform. Theory}, 34(4):605--612, 1988.

\bibitem{DaWa}
M.~A. Davenport and M.~B. Wakin.
\newblock Compressive sensing of analog signals using discrete prolate spheroidal sequences.
\newblock {\em Appl. Comput. Harmon. Anal.}, 33(3):438--472, 2012.

\bibitem{DaSe}
G.~David and S.~Semmes.
\newblock {\em Analysis of and on {U}niformly {R}ectifiable {S}ets}, volume~38 of {\em Mathematical Surveys and Monographs}.
\newblock American Mathematical Society, Providence, RI, 1993.

\bibitem{dMFeNo}
F.~{De Mari}, H.~G. Feichtinger, and K.~Nowak.
\newblock Uniform eigenvalue estimates for time-frequency localization operators.
\newblock {\em J. London Math. Soc. (2)}, 65(3):720--732, 2002.

\bibitem{operator-stft}
M.~D\"orfler, F.~Luef, H.~McNulty, and E.~Skrettingland.
\newblock Time-frequency analysis and coorbit spaces of operators.
\newblock {\em J. Math. Anal. Appl.}, 534:128058, 2024.

\bibitem{MR2452833}
M.~Engli{\v s}.
\newblock Toeplitz operators and localization operators.
\newblock {\em Trans. Amer. Math. Soc.}, 361(2):1039--1052, 2009.

\bibitem{MR3469849}
I.~M. Gelfand and G.~E. Shilov.
\newblock {\em Generalized {F}unctions. {V}ol. 2}.
\newblock AMS Chelsea Publishing, Providence, RI, 2016.
\newblock Spaces of fundamental and generalized functions, Translated from the 1958 Russian original [MR0106409] by M. D. Friedman, A. Feinstein and C. P. Peltzer, Reprint of the 1968 English translation [MR0230128].

\bibitem{grbook}
K.~Gr{\"o}chenig.
\newblock {\em Foundations of {T}ime-{F}requency {A}nalysis}.
\newblock Applied and Numerical Harmonic Analysis. Birkh{\"a}user Boston, Inc., Boston, MA, 2001.

\bibitem{wiener}
K.~Gr\"ochenig and M.~Leinert.
\newblock Wiener's lemma for twisted convolution and {G}abor frames.
\newblock {\em J. Amer. Math. Soc.}, 17(1), 2003.

\bibitem{MR2027858}
K.~Gr{\"o}chenig and G.~Zimmermann.
\newblock Spaces of test functions via the {STFT}.
\newblock {\em J. Funct. Spaces Appl.}, 2(1):25--53, 2004.

\bibitem{Ha}
S.~Halvdansson.
\newblock Empirical plunge profiles of time-frequency localization operators.
\newblock {\em Appl. Comput. Harmon. Anal.}, 81:101825, 2026.

\bibitem{HK}
J.~Heinonen and P.~Koskela.
\newblock Quasiconformal maps in metric spaces with controlled geometry.
\newblock {\em Acta Math.}, 181(1):1--61, 1998.

\bibitem{MR1408902}
E.~Hern\'{a}ndez and G.~Weiss.
\newblock {\em A {F}irst {C}ourse on {W}avelets}.
\newblock Studies in Advanced Mathematics. CRC Press, Boca Raton, FL, 1996.

\bibitem{HuIsMa}
K.~Hughes, A.~Israel, and A.~Mayeli.
\newblock On the eigenvalue distribution of spatio-spectral limiting operators in higher dimensions, {II}.
\newblock {\em J. Fourier Anal. Appl.}, 31(51), 2025.

\bibitem{HuIsMa2}
K.~Hughes, A.~Israel, and A.~Mayeli.
\newblock Wave packets and eigenvalue estimates for limiting operators on the disk.
\newblock {\em arXiv:2601.21224}, 2026.

\bibitem{israel15}
A.~Israel.
\newblock The eigenvalue distribution of time-frequency localization operators.
\newblock {\em arXiv:1502.04404}, 2015.

\bibitem{IsMa}
A.~Israel and A.~Mayeli.
\newblock On the eigenvalue distribution of spatio-spectral limiting operators in higher dimensions.
\newblock {\em Appl. Comput. Harmon. Anal.}, 70:Paper No. 101620, 2024.

\bibitem{KaRoDa}
S.~Karnik, J.~Romberg, and M.~A. Davenport.
\newblock Improved bounds for the eigenvalues of prolate spheroidal wave functions and discrete prolate spheroidal sequences.
\newblock {\em Appl. Comput. Harmon. Anal.}, 55:97--128, 2021.

\bibitem{MR4711850}
A.~Kulikov.
\newblock Exponential lower bound for the eigenvalues of the time-frequency localization operator before the plunge region.
\newblock {\em Appl. Comput. Harmon. Anal.}, 71:Paper No. 101639, 8, 2024.

\bibitem{kul2}
A.~Kulikov.
\newblock Sharp estimates for eigenvalues of localization operators before the plunge region.
\newblock {\em arXiv:2603.07407}, 2026.

\bibitem{kulikov2026sharp}
A.~Kulikov and M.~D. Larsen.
\newblock Sharp estimates for eigenvalues of localization operators with applications to area laws.
\newblock {\em arXiv preprint: 2603.23832}, 2026.

\bibitem{La}
P.~Lahti.
\newblock A {F}ederer-style characterization of sets of finite perimeter on metric spaces.
\newblock {\em Calc. Var. Partial Differential Equations}, 56(5):Paper No. 150, 2017.

\bibitem{La2}
H.~J. Landau.
\newblock On {S}zeg{\H o}'s eigenvalue distribution theorem and non-{H}ermitian kernels.
\newblock {\em J. Analyse Math.}, 28:335--357, 1975.

\bibitem{LaPo2}
H.~J. Landau and H.~O. Pollak.
\newblock Prolate spheroidal wave functions, {F}ourier analysis and uncertainty. {II}.
\newblock {\em Bell System Tech. J.}, 40:65--84, 1961.

\bibitem{LaPo3}
H.~J. Landau and H.~O. Pollak.
\newblock Prolate spheroidal wave functions, {F}ourier analysis and uncertainty. {III}. {T}he dimension of the space of essentially time- and band-limited signals.
\newblock {\em Bell System Tech. J.}, 41:1295--1336, 1962.

\bibitem{LaWi}
H.~J. Landau and H.~Widom.
\newblock Eigenvalue distribution of time and frequency limiting.
\newblock {\em J. Math. Anal. Appl.}, 77(2):469--481, 1980.

\bibitem{accumulated-cohen}
F.~Luef and E.~Skrettingland.
\newblock On accumulated {C}ohen’s class distributions and mixed-state localization operators.
\newblock {\em Constr. Approx.}, 52:31–64, 2020.

\bibitem{MR}
F.~Marceca and J.~L. Romero.
\newblock Spectral deviation of concentration operators for the short-time {F}ourier transform.
\newblock {\em Studia Math.}, 270(2):145--173, 2023.

\bibitem{MR2}
F.~Marceca and J.~L. Romero.
\newblock Improved discrepancy for the planar {C}oulomb gas at low temperatures.
\newblock {\em J. Anal. Math.}, 157(1):113--153, 2025.

\bibitem{eigenfourier}
F.~Marceca, J.~L. Romero, and M.~Speckbacher.
\newblock Eigenvalue estimates for {F}ourier concentration operators on two domains.
\newblock {\em Arch. Ration. Mech. Anal.}, 248(3):Paper No. 35, 2024.

\bibitem{MaRoTo}
J.~M. Maz\'{o}n, J.~D. Rossi, and J.~J. Toledo.
\newblock {\em Nonlocal {P}erimeter, {C}urvature and {M}inimal {S}urfaces for {M}easurable {S}ets}.
\newblock Frontiers in Mathematics. Birkh\"{a}user/Springer, Cham, 2019.

\bibitem{Mi}
M.~Miranda~Jr.
\newblock Functions of bounded variation on ``good'' metric spaces.
\newblock {\em J. Math. Pures Appl. (9)}, 82(8):975--1004, 2003.

\bibitem{MR3433287}
J.~P. Oldfield.
\newblock Two-term {S}zeg{\H o} theorem for generalised anti-{W}ick operators.
\newblock {\em J. Spectr. Theory}, 5(4):751--781, 2015.

\bibitem{Os}
A.~Osipov.
\newblock Certain upper bounds on the eigenvalues associated with prolate spheroidal wave functions.
\newblock {\em Appl. Comput. Harmon. Anal.}, 35(2):309--340, 2013.

\bibitem{rkhs}
V.~Paulsen and M.~Raghupathi.
\newblock {\em An {I}ntroduction to the {T}heory of {R}eproducing {K}ernel {H}ilbert {S}paces}.
\newblock Cambridge University Press, 2016.

\bibitem{eirik}
E.~Skrettingland.
\newblock Equivalent norms for modulation spaces from positive {C}ohen’s class distributions.
\newblock {\em J. Fourier Anal. Appl.}, 28(2), 2022.

\bibitem{SlPo}
D.~Slepian and H.~O. Pollak.
\newblock Prolate spheroidal wave functions, {F}ourier analysis and uncertainty. {I}.
\newblock {\em Bell System Tech. J.}, 40:43--63, 1961.

\bibitem{So}
A.~V. Sobolev.
\newblock Pseudo-differential operators with discontinuous symbols: {W}idom's conjecture.
\newblock {\em Mem. Amer. Math. Soc.}, 222(1043):vi+104, 2013.

\bibitem{SoTr}
J.~Soria and P.~Tradacete.
\newblock The least doubling constant of a metric measure space.
\newblock {\em Ann. Acad. Sci. Fenn. Math.}, 44(2):1015--1030, 2019.

\bibitem{Th}
D.~Thomson.
\newblock Spectrum estimation and harmonic analysis.
\newblock {\em Proceedings of the IEEE}, 70(9):1055--1096, 1982.

\bibitem{valtesis}
L.~Valentini.
\newblock Spectral deviation of {T}oeplitz operators acting on reproducing kernel {H}ilbert spaces.
\newblock Master's thesis, Universit{\`a} degli studi di Genova, 2024.

\bibitem{Wi}
H.~Widom.
\newblock On a class of integral operators on a half-space with discontinuous symbol.
\newblock {\em J. Funct. Anal.}, 88(1):166--193, 1990.

\bibitem{ZeMe}
T.~Zemen and C.~F. Mecklenbr\"{a}uker.
\newblock Time-variant channel estimation using discrete prolate spheroidal sequences.
\newblock {\em IEEE Trans. Signal Process.}, 53(9):3597--3607, 2005.

\end{thebibliography}

\end{document}